\documentclass[11pt]{article}
\usepackage{amsfonts}
\usepackage{amsmath,amssymb}
\usepackage[dvips]{graphicx}
\usepackage{algorithmicx}
\usepackage{algpseudocode}
\usepackage{yhmath}
\usepackage{qtree}
\usepackage{xcolor}
\usepackage{epsfig}
\usepackage{lineno}
\usepackage{mathtools}
\usepackage{amsmath,bm}
\usepackage{amsthm}
\usepackage[T1]{fontenc}
\usepackage[utf8]{inputenc}
\usepackage{enumitem}
\usepackage{multirow}
\usepackage{makecell}
\usepackage{algorithm}

\usepackage{authblk}
\usepackage{bigints}

\usepackage[toc]{appendix}
\numberwithin{equation}{section}

\usepackage{subfigure}
\usepackage{verbatim}
\usepackage{graphicx}
\usepackage{geometry}
\usepackage{epsfig}
\usepackage{lineno,hyperref}
\usepackage[many]{tcolorbox}
\usetikzlibrary{shadows}
\usepackage{caption}
\usepackage{hyperref}

\newtcolorbox{shadedbox}{
  breakable,
  enhanced jigsaw,
  colback=white,
}

\newtheorem{theorem}{Theorem}[section]
\newtheorem{lemma}{Lemma}[section]
\newtheorem{remark}{Remark}[section]
\newtheorem{proposition}{Proposition}[section]
\newtheorem{definition}{Definition}[section]

\makeatletter
\def\BState{\State\hskip-\ALG@thistlm}
\makeatother

\newcommand{\lc}{\mathrel{\raise2pt\hbox{${\mathop<\limits_{\raise1pt\hbox
{\mbox{$\sim$}}}}$}}}
\newcommand{\gc}{\mathrel{\raise2pt\hbox{${\mathop>\limits_{\raise1pt\hbox
{\mbox{$\sim$}}}}$}}}
\newcommand{\ec}{\mathrel{\raise1pt\hbox{${\mathop=\limits_{\raise2pt\hbox
{\mbox{$\sim$}}}}$}}}

\renewcommand{\theequation}{\arabic{section}.\arabic{equation}}

\renewcommand{\thefigure}{\arabic{figure}}%

\newcommand{\red}{\color{red}}

\usepackage{fancyhdr}
\begin{document}

\title{Convergence analysis of a solver for the linear Poisson--Boltzmann model}
\author[1]{Xuanyu Liu}
\author[2]{Yvon Maday}
\author[3]{Chaoyu Quan}
\author[4]{Hui Zhang}

\date{}
\affil[1]{\small School of Mathematical Sciences, Beijing Normal University,
Beijing, China
(\href{mailto:xyliu9535@mail.bnu.edu.cn}{xyliu9535@mail.bnu.edu.cn}).}
\affil[2]{\small  Sorbonne Universit\'e, CNRS, Universit\'e de Paris,
Laboratoire Jacques-Louis Lions (LJLL), F-75005 Paris, France
(\href{yvon.maday@sorbonne-universite.fr}{yvon.maday@sorbonne-universite.fr})}
\affil[3]{\small Division of Mathematics, School of Science and Engineering, The Chinese University of Hong Kong, Shenzhen, 518172, Guangdong, China
(\href{quanchaoyu@cuhk.edu.cn}{quanchaoyu@cuhk.edu.cn}).}
\affil[4]{\small Department of Mathematical Sciences, Xi'an Jiaotong-Liverpool
University, Suzhou, Jiangsu, China
(\href{mailto:Hui.Zhang@xjtlu.edu.cn}{hui.zhang@xjtlu.edu.cn}).}

\maketitle

\begin{abstract}
This work investigates the convergence of a domain decomposition method for the Poisson-Boltzmann model that can be formulated as an interior-exterior transmission problem. To study its convergence, we introduce an interior-exterior constant providing an upper bound of the $L^2$ norm of any harmonic function in the interior, and establish a spectral equivalence for related Dirichlet-to-Neumann operators to estimate the spectrum of interior-exterior iteration operator. This analysis is nontrivial due to the unboundedness of the exterior subdomain, which distinguishes it from the classical analysis of the Schwarz alternating method with nonoverlapping bounded subdomains. It is proved that for the linear Poisson--Boltzmann solvent model in reality, the convergence of interior-exterior iteration is ensured when the relaxation parameter lies between $0$ and $2$. This convergence result interprets the good performance of ddLPB method developed in \cite{quan2019domain} where the relaxation parameter is set to $1$. Numerical simulations are conducted to verify our convergence analysis and to investigate the optimal relaxation parameter for the interior-exterior iteration.
\end{abstract}


\section{Introduction}

In this work, we study an interior-exterior domain decomposition method for a widely-used solvent model, the Poisson--Boltzmann (PB) solvent model \cite{yoon1990boundary,nicholls1991rapid,tomasi-2005}. Specifically, we consider the linear Poisson--Boltzmann (LPB) equation, which takes the form
\begin{equation}\label{eq:lpb}
\begin{array}{r@{}l}
\left\{
\begin{aligned}
 & - \Delta \psi(\mathbf x)   = 4\pi \varepsilon_1^{-1}\rho_{\rm m}(\mathbf x) &&
\mbox{in $\Omega$},\\
& - \Delta \psi(\mathbf x) + \kappa^2 \psi(\mathbf x)  = 0 && \mbox{in
$\Omega^{\mathsf c} \coloneqq \mathbb R^3\backslash \Omega$},
\end{aligned}
\right.
\end{array}
\end{equation}
where $\psi$ is the electrostatic potential, $\rho_{\rm m}$ is the solute's charge distribution within the solute cavity $\Omega$, $\kappa>0$ is the modified Debye--H\"uckel constant, and $\varepsilon_1$ is the dielectric constant of solute.
The LPB equation also includes two transmission conditions at the solute-solvent interface $\Gamma$:
\begin{equation}\label{eq:coupling0}
\begin{array}{r@{}l}
\left\{
\begin{aligned}
\psi|_{ \Omega} & = \psi|_{\Omega^{\mathsf c}}  && \mbox{on $\Gamma$},\\
\varepsilon_1 \partial_{\mathbf n} \psi|_{\Omega} & = \varepsilon_2
\partial_{\mathbf n}  \psi|_{\Omega^{\mathsf c}}  && \mbox{on $\Gamma$},
\end{aligned}
\right.
\end{array}
\end{equation}
where $\mathbf n$ is
the unit normal vector pointing outward from $\Omega$ to $\Omega^{\mathsf c}$ and $\varepsilon_2$ is the dielectric constant of solvent.
We further assume that the potential $\psi$ behaves like $1/|\mathbf x|$ as $|\mathbf x|$ tends to infinity, and $\Omega$ is a bounded Lipschitz domain.

It is worth noting that when $\kappa = 0$, the LPB model reduces to the polarizable continuum model (PCM), which is also commonly used in solvation energy calculations \cite{cammi2007continuum}. On the other hand, when $\kappa = \infty$, the solvent is treated as a perfect conductor, and the LPB model simplifies to a Poisson equation defined only on $\Omega$. This simplified model is known as the conductor-like screening model (COSMO) \cite{klamt1993cosmo}.

Generally speaking, there are three types of widely-used methods to solve the PB model: boundary element method (BEM), finite difference method (FDM), and finite element method (FEM). In the BEM, the LPB equation is recast as integral equations defined on the two-dimensional solute-solvent interface. Then surface meshes are generated (for example, using MSMS \cite{sanner-1996} or NanoShaper \cite{zhang2006quality}) to solve the integral equations. The BEM is efficient to solve the LPB equation, especially when acceleration techniques are used, such as the fast multipole method \cite{boschitsch2002fast,altman2009accurate,bajaj2011efficient} and the hierarchical ``treecode'' technique \cite{geng2013treecode,chen2018preconditioning}. The PB-SAM solver developed by Head-Gordon et al. \cite{lotan2006analytical,yap2010new,yap2013calculating} discretizes the solute-solvent interface (such as the van der Waals surface) with grid points on atomic spheres like a collocation method and solves the associated linear system by use of the fast multipole method. 
In general, it is difficult to apply BEM to the nonlinear Poisson--Boltzmann (NPB) equation. To solve the NPB equation, there have been some efficient FDM solvers, including UHBD \cite{madura1995electrostatics}, DelPhi \cite{li2012delphi}, MIBPB \cite{chen2011mibpb}, and APBS \cite{baker2001electrostatics,dolinsky2007pdb2pqr,jurrus2018improvements}. In particular, the APBS developed by Baker, Holst, McCammon et al. is popular and has many useful options. The FEM provides in general more flexibility for mesh refinement, more analysis of convergence, and more selections of linear solvers. A rigorous solution and approximation theory of the FEM for the PB equation has been established in \cite{chen2007finite}. The adaptive FEM developed by Holst et al. has tackled some important issues of the PB equation \cite{holst2000adaptive,baker2000adaptive}. Note that the BEM and FEM are mesh-dependent where the mesh generation itself might be complicated. The FDM depends on the grid, and the grid refinement in three dimensions could result in a dramatic increase of computational cost.

A special domain decomposition (dd) method for the COSMO model in van der Waals (VDW) cavity was developed in 2013 by Canc\`es, Maday, and Stamm \cite{cances2013domain}. This method, called ddCOSMO, is meshless and can perform up to three orders of magnitude faster than the equivalent algorithm in Gaussian, as shown in \cite{lipparini2013fast,lipparini2014quantum,lipparini2014quantum2}.
The ddCOSMO method decomposes the solute cavity $\Omega$ into balls and transforms the original problem into a group of easy-to-solve Laplace subproblems defined in these balls. Since its development, the ddCOSMO solver has been generalized to the PCM and LPB models, respectively, resulting in the ddPCM \cite{stamm2016new} and ddLPB methods \cite{quan2019domain}. Additionally, the method has been extended to the PCM with smooth interface \cite{Quan2016760,quan2018domain}.

Regarding the convergence analysis, the ddCOSMO method exhibits linear scaling with respect to the number of atoms and the first results on these scaling properties in a simplified setting can be found in a series of works by Ciaramella and Gander \cite{ciaramella2017analysis,ciaramella2018analysis,ciaramella2018analysis2}. Recently, Reusken and Stamm proposed specific geometrical descriptors and demonstrated how the convergence rate of the ddCOSMO solver depends on them \cite{reusken2021analysis}.
In the ddLPB method, the transmission problem defined by the interior Poisson equation and the exterior screening Poisson equation are solved 
iteratively by solving only two interior problems defined by the (screening) Poisson equations, with the same Dirichlet boundary condition on $\partial \Omega$ obtained from the previous iteration. After each iteration, this Dirichlet boundary condition is then updated by the screening single layer potential 
based on the resolved solutions. This process is repeated until a stopping criterion is reached. Despite the good numerical performance of ddLPB method, its convergence analysis remains unknown.

In the classical works \cite{lions1988schwarz, lions1990schwarz}, Lions studied Schwarz alternating methods for overlapping and non-overlapping bounded subdomains, respectively, which are apparently different from the ddLPB method. 
The LPB equation can be seen as an interior-exterior transmission problem, and the ddLPB method deals with interior-exterior coupling through arbitrary interface without boundary integral equation.
The domain decomposition method closest to the ddLPB may be Yu's Dirichlet-Neumann method \cite{yu1997domain}. In his work, Yu gives the convergence analysis for Poisson equation on an unbounded domain with a closed curve $\Gamma_0$ as its boundary. He divides the unbounded domain into two parts with a circle interface containing $\Gamma_0$, so that the exterior subproblem defined outside a circle becomes easy to solve.
However, in our case, the parameters $\varepsilon$ and $\kappa$ has a jump on a  general complicated interface $\Gamma$ (not a circle), so that there is no explicit formula of solution of the exterior subproblem. 

In this work, we establish the convergence of the interior-exterior iteration for the ddLPB method, which is a nonoverlapping Schwarz alternating method with a bounded interior and an unbounded exterior.
We first prove that $T_{\rm e} + T_{\rm c}$ is spectrally equivalent to $T_{\rm r} + T_{\rm c}$, where $T_{\rm r},~ T_{\rm c},~T_{\rm e}$ are Dirichlet-to-Neumann (DtN) operators defined in \eqref{eq:Tr}--\eqref{eq:Te}. 
The proof is based on an $L^2$-norm estimate for arbitrary harmonic function (see Lemma \ref{lem:1}) with the definition of an interior-exterior Sobolev constant.
We show that this spectral equivalence can result in convergence by equipping the Hilbert space $H^{\frac12}(\Gamma)$ with a new inner product, and consequently the convergence of ddLPB is ensured if the relaxation parameter satisfies $0<\alpha<2$ for realistic implicit solvent models. This interprets the good numerical performance of ddLPB solver in \cite{quan2019domain}, where $\alpha$ is set to $1$. 

In the following content, $C$ is a generic uniform constant in different inequalities. When it is required, we will emphasize its dependency on the parameters and the domain in the governing partial differential equation. 
  
 The article is organized as follows. 
After shortly introducing the original ddLPB method, we first reinterpret the ddLPB method as a preconditioned Richardson iteration for updating the potential on interface and generalize the original ddLPB method by adding a relaxation parameter $\alpha\in \mathbb R$ in Section \ref{sect2}.
Then, in Section \ref{sect3}, we prove the convergence of the ddLPB method by establishing the spectral equivalence of 
the operator and its preconditioner.
To achieve this, we introduce a new inner product on the Hilbert space $H^{\frac12}(\Gamma)$, and establish an $L^2$-norm estimate for any harmonic function, using an interior-exterior constant $C_{\rm ie}^r$. We then use this estimate to prove the desired spectral equivalence, which in turn guarantees the convergence of the ddLPB method for a suitable choice of the relaxation parameter $\alpha$.
In Section \ref{sec:num}, we present numerical simulations to verify the convergence of the ddLPB method, and provide some concluding remarks in the last section.
  
\section{An interior-exterior domain decomposition method}\label{sect2}

We first transform the LPB equation to two coupled subproblems defined in the bounded domain $\Omega$. Then we introduce a preconditioner for the equation associated with boundary potential and derive a Richardson iteration. 
This iteration yields a generalized ddLPB method. In comparison to the original approach \cite{quan2019domain}, a relaxation parameter is used to accelerate the convergence.

\subsection{Problem transformation}
Let $\psi_0$ be the potential generated by $\rho_{\rm m}$ in vacuum that satisfies the following equation:
\begin{equation}
- \Delta \psi_0 = 4\pi \varepsilon_1^{-1} \rho_{\rm m} \quad \mbox{in }\mathbb R^3.
\end{equation}
The potential $\psi_0$ can be obtained explicitly by calculating the following integral:
\begin{equation}
\psi_0(\mathbf x) = \int_{\mathbb R^3} \frac{\rho_{\rm m}(\mathbf
y)}{\varepsilon_1\left|\mathbf x - \mathbf y\right|} \,{\rm d} \mathbf y, \quad
\mathbf x\in \mathbb R^3.
\end{equation}
For simplicity, we assume that $\psi_0$ is already given.
Then, \eqref{eq:lpb} can be recast equivalently to two coupled partial differential equations (PDEs) of reaction potential $\psi_{\rm r}$ and extended potential $\psi_{\rm e}$, both defined in $\Omega$ (see \cite{quan2019domain} for the derivation): 
\begin{equation}\label{eq:lpb2}
\begin{array}{r@{}l}
\left\{
\begin{aligned}
- \Delta \psi_{\rm r} & = 0  && \mbox{in $\Omega$},\\
-\Delta \psi_{\rm e} + \kappa^2 \psi_{\rm e}& = 0 && \mbox{in $\Omega$},
\end{aligned}
\right.
\end{array}
\end{equation}
subject to two coupling conditions
\begin{equation}\label{eq:coupling}
\begin{array}{r@{}l}
\left\{
\begin{aligned}
 \psi_{\rm r} & = \psi_{\rm e} - \psi_0 && \mbox{on $\Gamma$},\\
\psi_{\rm e} & = \mathcal S_{\kappa} \left( \partial_{\mathbf n} \psi_{\rm e} -
{\varepsilon_1}{\varepsilon_2^{-1}} \partial_{\mathbf n} \left(\psi_{\rm r}
+\psi_0\right) \right) && \mbox{on $\Gamma$},
\end{aligned}
\right.
\end{array}
\end{equation}
where $\psi_{\rm r} = \psi|_\Omega - \psi_0$ is the reaction potential and
$\psi_{\rm e}$ is an extended potential, both defined in $\Omega$.
We define the invertible single-layer operator $\mathcal S_\kappa: H^{-\frac 12}(\Gamma)\rightarrow H^{\frac
12}(\Gamma)$ as follows:
\begin{equation}\label{eq:S_kappa0}
\left(\mathcal S_\kappa \sigma_\Gamma\right) (\mathbf x) \coloneqq \int_\Gamma
\frac{\exp\left(-\kappa\left|\mathbf x-\mathbf s'\right|\right) \sigma_\Gamma(\mathbf
s')}{4\pi\left|\mathbf x - \mathbf s'\right|} \, \mathrm d \mathbf s' \quad \text{for all }
\sigma_\Gamma\in H^{-\frac 1 2}(\Gamma) \text{ and } \mathbf x\in \Gamma.
\end{equation}
In the special case where $\kappa=0$, the homogeneous screened Poisson (HSP) equation (\ref{eq:lpb2}) reduces to the Laplace equation, and we denote the corresponding single-layer operator as $\mathcal S_0$.

The original problem \eqref{eq:lpb}, defined in $\mathbb R^3$, is equivalently transformed into two subproblems \eqref{eq:lpb2}, which are defined in $\Omega$ and coupled at the interface $\Gamma$ by the conditions \eqref{eq:coupling}. This transformation recasts the original interior-exterior transmission problem as an interior-interior transmission problem.

\subsection{Richardson iteration}

The LPB problem \eqref{eq:lpb2}--\eqref{eq:coupling} can be seen as an interface problem: find $g\in H^{\frac 1 2}(\Gamma)$, satisfying
\begin{equation}
T_{\rm r}\left(g-\psi_{0}\right) + \varepsilon_1 \partial_{\mathbf n} \psi_0 +
T_{\rm c} \, g = 0,
\end{equation}
that is,
\begin{equation}\label{eq:interface}
 \left(T_{\rm r} + T_{\rm c}\right) g = -\varepsilon_1\partial_{\mathbf n}
\psi_0 + T_{\rm r} \psi_0,
\end{equation}
where $T_{\rm r}: g \mapsto \varepsilon_1 \partial_{\mathbf n} u_{\rm r}$ is a DtN operator in the bounded domain $\Omega$ that satisfies
\begin{equation}\label{eq:Tr}
\begin{array}{r@{}l}
\left\{
\begin{aligned}
- \Delta u_{\rm r} & = 0  && \mbox{in $\Omega$},\\
u_{\rm r}& = g && \mbox{on $\Gamma$},
\end{aligned}
\right.
\end{array}
\end{equation}
and $T_{\rm c}: g \mapsto -\varepsilon_2 \partial_{\mathbf n}u_{\rm c}$ is another DtN operator in the exterior domain $\Omega^{\mathsf c}$, satisfying
\begin{equation}\label{eq:Tc}
\begin{array}{r@{}l}
\left\{
\begin{aligned}
- \Delta u_{\rm c} + \kappa^2 u_{\rm c}& = 0  && \mbox{in $\Omega^{\mathsf
c}$},\\
u_{\rm c}& = g && \mbox{on $\Gamma$},\\
u_{\rm c}& \rightarrow O(|\mathbf x|^{-1})&& \mbox{as } |\mathbf x|\rightarrow \infty.
\end{aligned}
\right.
\end{array}
\end{equation}
Furthermore, we define the following DtN operator $T_{\rm e}: g \mapsto
\varepsilon_2 \partial_{\mathbf n}u_{\rm e}$, satisfying 
\begin{equation}\label{eq:Te}
\begin{array}{r@{}l}
\left\{
\begin{aligned}
- \Delta u_{\rm e} + \kappa^2 u_{\rm e}& = 0  && \mbox{in $\Omega$},\\
u_{\rm e}& = g && \mbox{on $\Gamma$}.
\end{aligned}
\right.
\end{array}
\end{equation}
For any $g\in H^{\frac 1 2}(\Gamma)$, the corresponding surface charge density $\sigma\in H^{-\frac 1 2}(\Gamma)$ satisfies
\begin{equation}
\begin{aligned}
\sigma = {\varepsilon_2^{-1}} \left(T_{\rm e} + T_{\rm c} \right) g, 
\end{aligned}
\end{equation}
which yields from \cite[Theorem 3.3.1]{Sauter2011} that
\begin{equation}
g = \varepsilon_2^{-1}\mathcal S_\kappa \left(T_{\rm e}+T_{\rm c}\right) g\quad\forall g\in H^{\frac12}(\Gamma),
\end{equation}
or again, $T_{\rm e}+ T_{\rm c}$ is invertible and
\begin{equation}\label{eq:inv_Ska}
\left(T_{\rm e}+ T_{\rm c}\right)^{-1} = \varepsilon_2^{-1}\mathcal S_\kappa :\, H^{-\frac 1 2}(\Gamma) \rightarrow H^{\frac 1 2}(\Gamma).
\end{equation}
We now take $\left(T_{\rm e}+ T_{\rm c}\right)^{-1}$ as a preconditioner for \eqref{eq:interface}, to obtain 
\begin{equation}
\left(T_{\rm e}+ T_{\rm c}\right)^{-1}\left[- \left(T_{\rm r} + T_{\rm c}\right) g +\left(-\varepsilon_1\partial_{\mathbf
n} \psi_0 + T_{\rm r} \psi_0\right) \right] = 0.
\end{equation}
The Richardson iteration for solving this equation is
\begin{equation}\label{eq:iteration}
\begin{array}{r@{}l}
\begin{aligned}
g^{k+1} & =g^k +  \alpha\,\left(T_{\rm e}+ T_{\rm c}\right)^{-1} \left[ -
\left(T_{\rm r} + T_{\rm c}\right) g^k + \left(-\varepsilon_1\partial_{\mathbf
n} \psi_0 + T_{\rm r} \psi_0\right) \right]\\
& =\left[I-\alpha \left(T_{\rm e}+ T_{\rm c}\right)^{-1} 
\left(T_{\rm r} + T_{\rm c}\right) \right] g^k +  \alpha\,\left(T_{\rm e}+ T_{\rm c}\right)^{-1} \left(-\varepsilon_1\partial_{\mathbf
n} \psi_0 + T_{\rm r} \psi_0\right) \\
 & = (1-\alpha) \, g^k + \alpha\, \left(T_{\rm e}+ T_{\rm c}\right)^{-1} 
\left[  T_{\rm e} \, g^k-  T_{\rm r} \left(g^k -\psi_0\right)
-\varepsilon_1\partial_{\mathbf n} \psi_0  \right] \\
& = (1-\alpha) \, g^k + \alpha\, \varepsilon_2^{-1}\mathcal S_\kappa \left[
\varepsilon_2 \partial_{\mathbf n} \psi_{\rm e}^k - 
{\varepsilon_1}\partial_{\mathbf n} \psi_{\rm r}^{k} 
-\varepsilon_1\partial_{\mathbf n} \psi_0  \right] \\
& = (1-\alpha) \, g^k + \alpha  \mathcal S_\kappa \left[  \partial_{\mathbf n}
\psi_{\rm e}^k -  {\varepsilon_1}{\varepsilon_2^{-1}}\partial_{\mathbf n}
\left( \psi_{\rm r}^{k} +  \psi_0 \right) \right],
\end{aligned}
\end{array}
\end{equation}
where $\alpha\in \mathbb R$ is an appropriate scalar relaxation parameter. Here, $I$ is the identity operator.


\subsection{Generalized ddLPB method}
The original ddLPB method in \cite{quan2019domain} updates the potential on the solute-solvent interface at each external iteration. We propose a generalization of the ddLPB method, using the Richardson iteration \eqref{eq:iteration}. The new procedure still involves five steps, as shown in Algorithm \ref{algo}. Unlike the original ddLPB method where $\alpha$ is fixed to 1 in step 4, the relaxation parameter $\alpha\in \mathbb R$ is used that controls the step size.

It was observed in the numerical experiments conducted in \cite{quan2019domain} that the interfacial iteration \eqref{eq:iteration} in the ddLPB procedure converges quickly with respect to the iteration number $k$ when $\alpha=1$. We will demonstrate that the generalized method converges if $\alpha$ satisfies a simple restriction. The choice of this parameter can impact the convergence rate. Therefore, selecting an appropriate value of $\alpha$ can accelerate the algorithm by reducing the number of external iterations required.

\begin{algorithm}[!ht]
\caption{Generalized ddLPB method}
\begin{enumerate}
\item Let $g^0$ be an initial guess of $\psi_{\rm e}|_{\Gamma}$ and set the iteration number to $k=1$.
\item 
Solve the Laplace equation 
\begin{equation}
\begin{array}{r@{}l}
\left\{
\begin{aligned}
-\Delta\psi_{\rm r}^k&= 0 && \mbox{in } \Omega ,\\
\psi_{\rm r}^k & = g^{k-1} - \psi_{0} && \mbox{on }\Gamma,
\end{aligned}
\right. \label{eq:dd_laplace}
\end{array}
\end{equation}
and obtain its Neumann boundary trace $\partial_{\mathbf n} \psi_{\rm r}^k$ on
$\Gamma$.
\item
Solve the screened Poisson equation
\begin{equation}
\begin{array}{r@{}l}
\left\{
\begin{aligned}
- \Delta\psi_{\rm e}^k + \kappa^2 \psi_{\rm e}^k& = 0 && \mbox{in } \Omega ,\\
\psi_{\rm e}^k & = g^{k-1} && \mbox{on }\Gamma,
\end{aligned}
\right.
\end{array}\label{eq:dd_screen}
\end{equation}
and obtain its Neumann boundary trace $\partial_{\mathbf n}\psi_{\rm e}^k$ on
$\Gamma$.
\item 
Update the Dirichlet boundary condition 
\begin{equation}\label{eq:update}
g^k = (1-\alpha)\, g^{k-1} + \alpha \, \mathcal {S}_{\kappa}
\left[\partial_{\mathbf n}\psi_{\rm e}^k -
{\varepsilon_1}{\varepsilon_2^{-1}}\partial_{\mathbf n}\left(\psi_0+\psi_{\rm
r}^k \right)\right]
\end{equation}
with an appropriate constant $\alpha\in \mathbb R$.
\item 
Repeat Step 2--5 until some convergence tolerance is reached.
\end{enumerate}
\label{algo}
\end{algorithm}

\section{Convergence analysis}\label{sect3}

In this section, we propose a concept of interior-exterior constant
to derive the spectral equivalence between $T_{\rm e} + T_{\rm c}$ and $T_{\rm r} + T_{\rm c}$. Then, we show that such equivalence ensures the convergence of the Richardson iteration \eqref{eq:iteration} in the ddLPB solver when $0<\alpha<2$ and $\varepsilon_1\leq \varepsilon_2$. 

\subsection{Interior-exterior estimate}\label{sect3.3}

To establish the spectral equivalence of $T_{\rm e} + T_{\rm c}$ and $T_{\rm r} + T_{\rm c}$, we introduce the interior-exterior constant.

\begin{definition}[Interior-exterior constant]
 Given a bounded Lipschitz domain $\Omega$ with boundary $\partial \Omega$, the interior-exterior constant, for $r=0$ or $1$,  is defined by 
	\begin{equation}\label{def:CP}
	C_{\rm ie}^r := \sup_{g\neq 0\in H^{\frac 1 2}(\partial \Omega)  } \frac{ \|u_{\rm
r}\|_{H^r(\Omega)}^2 }{\|u_{\rm c1}\|_{H^1(\Omega^{\mathsf c})}^2},
	\end{equation}
	where $u_{\rm r}$ and $ u_{\rm c1}$ are respectively the unique solutions to
\begin{equation}\label{eq:ur_1}
\left\{
\begin{aligned}
	- \Delta u_{\rm r} & = 0  && \mbox{in $\Omega$},\\
	u_{\rm r}& = g && \mbox{on $\partial \Omega$},
\end{aligned}
\right.
\end{equation}
and
\begin{equation}
	\left\{
	\begin{aligned}\label{eq:uc_1}
	- \Delta u_{\rm c1} +  u_{\rm c1}& = 0  && \mbox{in $\Omega^{\mathsf c}$},\\
	u_{\rm c1}& = g && \mbox{on $\partial \Omega$},\\
	u_{\rm c1}& \rightarrow O(|\mathbf x|^{-1})&& \mbox{as } |\mathbf x|\rightarrow \infty.
	\end{aligned}
	\right.
\end{equation}
\end{definition}
Note that the existence and uniqueness of $u_{\rm r}$ in \eqref{eq:ur_1} and $u_{\rm c1}$ in \eqref{eq:uc_1} are not difficult to be obtained from Lax--Milgram lemma \cite[Theorem 2.10.4, Theorem 2.10.7]{sauter2010BEM}. 
Using the regularity of the solution $u_{\rm r}$ \cite[Theorem 2.10.4]{sauter2010BEM} and the boundedness of the trace operator $\gamma_0: H^1(\Omega^{\mathsf c})\rightarrow H^{\frac12}(\partial \Omega)$ \cite[Theorem 2.6.8]{sauter2010BEM}, there exist two constants $C$ and $\widetilde C$ depending on $\Omega$ such that
\begin{equation}\label{eq:tracethm}
    \|u_{\rm r}\|_{L^2(\Omega)}\leq \|u_{\rm r}\|_{H^1(\Omega)}\leq C \|g\|_{H^{\frac12}(\partial \Omega)}\leq \widetilde C \| u_{\rm c1}\|_{H^1(\Omega^{\mathsf c})} \quad\forall g\in H^{\frac12}(\partial \Omega).
\end{equation}
This implies that $C_{\rm ie}^r$ are both well-defined and finite.

With the definition of interior-exterior constant, we have the following estimation of the solution to the Laplace equation \eqref{eq:Tr}.  


\begin{lemma}[Interior-exterior estimate]\label{lem:1}
Given any $ g\in H^{\frac 1 2}(\Gamma)$, suppose that $u_{\rm r}$ is the solution to
\eqref{eq:Tr}. 
Then we have
\begin{equation}
\|u_{\rm r}\|_{L^2(\Omega)}^2  \leq C_{\rm ie}^0 \, \varepsilon_2^{-1} \max\{1,\kappa^{-2}\} 
\left< T_{\rm c}\, g, g \right>.
\end{equation}
where $C_{\rm ie}^0$ is the interior-exterior constant defined in \eqref{def:CP} depending only on $\Omega$.
\end{lemma}
\begin{proof}
For $ g\in H^{\frac 1 2}(\Gamma)$, $u_{\rm c}$ is the solution to \eqref{eq:Tc} and $u_{\rm c1}$ is the solution to \eqref{eq:uc_1}. 
Multiplying \eqref{eq:Tc} with $u_{\rm c}$ and integrating over
$\Omega^{\mathsf c}$, we can obtain
\begin{equation}\label{eq:lem1_3}
\left< T_{\rm c}\, g, g \right>= \varepsilon_2 \int_{\Omega^{\mathsf c}} \left(
\left|\nabla u_{\rm c}\right|^2 + \kappa^2 u_{\rm c}^2
\right) \geq 0.
\end{equation} 
Further, multiplying with first $u_{\rm c1}$ and then $u_{\rm c}$ on both sides of $\eqref{eq:uc_1}$
and integrating over $\Omega^{\mathsf c}$, we have
\begin{equation}
\begin{aligned}
 \lVert u_{\rm c1} \rVert_{H^1(\Omega^{\mathsf c})}^2 
&= \int_\Gamma \left(\partial_{\mathbf n} u_{\rm c1}\right) u_{\rm c1} 
= \int_\Gamma \left(\partial_{\mathbf n} u_{\rm c1}\right) u_{{\rm c}} = \int_{\Omega^{\mathsf c}} \left( \nabla u_{\rm c1} \cdot \nabla u_{\rm c} + u_{\rm c1} u_{\rm c}\right)\\
& \leq \frac{1}{2} \int_{\Omega^{\mathsf c}} \left(\lvert \nabla u_{\rm c}  \rvert^2 + u_{\rm c} ^2 + \lvert\nabla u_{\rm c1} \rvert^2 + u_{\rm c1}^2\right),
\end{aligned}
\end{equation}
which implies that
\begin{equation}
\lVert u_{\rm c1} \rVert_{H^1(\Omega^{\mathsf c})}^2  \leq 
\lVert u_{\rm c} \rVert_{H^1(\Omega^{\mathsf c})}^2.
\end{equation}
Combining this inequality with \eqref{eq:lem1_3}, we estimate
\begin{equation}\label{eq36}
\lVert u_{\rm c1} \rVert_{H^1(\Omega^{\mathsf c})}^2  \leq 
\lVert u_{\rm c} \rVert_{H^1(\Omega^{\mathsf c})}^2\leq \varepsilon_2^{-1}  \max\{1,\kappa^{-2}\} \left< T_{\rm c}\, g, g
\right>.
\end{equation}
According to the Definition \eqref{def:CP}, we then obtain the stated result.
\end{proof}

The interior-exterior constant $C_{\rm ie}^r$ is important in our convergence analysis. However, the precise estimation is tough.
In the special case where $\Omega$ is a ball with radius $R$,
the interior-exterior constant can be estimated by $C_{\rm ie}^0\leq \frac{1}{3}R^2$ and $C_{\rm ie}^1\leq \frac{1}{3}R^2+1$.
See Appendix \ref{append} for the estimates.
But for a general domain, this is a challenging task. 

\subsection{Spectral equivalence  between \texorpdfstring{$T_{\rm e} + T_{\rm c}$}{Te+Tc} and \texorpdfstring{$T_{\rm r} + T_{\rm c}$}{Tr+Tc}}
With the newly-defined interior-exterior constant $C_{\rm ie}^r$ in subsection \ref{sect3.3}, we are ready to state and prove the desired spectral equivalence  between $T_{\rm e} + T_{\rm c}$ and $T_{\rm r} + T_{\rm c}$.

\begin{theorem}[Spectral equivalence and convergence]\label{thm:1}
  There exist two positive constants $C_1$ and $ C_2$ independent of
  $g$, such that $\forall g\in H^{\frac 1 2} \left(\Gamma\right)$,
\begin{equation}\label{eq:thm1}
C_1 \left< \left(T_{\rm e}+T_{\rm c}\right)g,g\right> \leq \left< \left(T_{\rm
r}+T_{\rm c}\right)g,g\right> \leq C_2 \left< \left(T_{\rm e}+T_{\rm
c}\right)g,g\right>,
\end{equation}
where 
  \begin{equation}\label{eq:C12}
    C_1= \min\left\{\frac{\varepsilon_1}{\varepsilon_2}, \frac{1}{1+  \max\{1,\kappa^{2}\} C_{\rm ie}^0}
\right\},~C_2=\max\left\{1,\frac{\varepsilon_1}{ \varepsilon_2}\right\}.
  \end{equation}
\end{theorem}
\begin{proof}
Multiplying \eqref{eq:Te} by $u_{\rm e}$ and integrating over $\Omega$, we
obtain
\begin{equation}\label{eq:thm1_1}
\left< T_{\rm e}\, g, g \right> = \varepsilon_2 \int_\Gamma 
\left(\partial_{\mathbf n} u_{\rm e}\right) u_{\rm e} = \varepsilon_2 
\int_\Omega \left( \left|\nabla u_{\rm e}\right|^2+ \kappa^2 u_{\rm
e}^2\right).
\end{equation}
Similarly, multiplying \eqref{eq:Tr} by $u_{\rm r}$ and integrating over
$\Omega$, we have 
\begin{equation}\label{eq:thm1_2}
\left< T_{\rm r}\, g, g \right>= \varepsilon_1 \int_\Omega \left|\nabla u_{\rm
r}\right|^2.
\end{equation}
Furthermore, multiplying \eqref{eq:Te} by $u_{\rm r}$ and integrating over
$\Omega$, we obtain
\begin{equation}\label{eq:thm1_3}
\left< T_{\rm e}\, g, g \right> =  \varepsilon_2 \int_\Gamma
\left(\partial_{\mathbf n} u_{\rm e}\right) u_{\rm r} = \varepsilon_2
\int_\Omega  \left(\nabla u_{\rm e} \cdot \nabla u_{\rm r} +\kappa^2 u_{\rm
e}u_{\rm r}\right).
\end{equation}
Similarly, multiplying \eqref{eq:Tr} by $u_{\rm e}$ and integrating over
$\Omega$, we have 
\begin{equation}\label{eq:thm1_4}
\left< T_{\rm r}\, g, g \right>= \varepsilon_1 \int_\Omega \nabla u_{\rm r}
\cdot \nabla u_{\rm e}.
\end{equation}

We now prove the second inequality of \eqref{eq:thm1}.
According to \eqref{eq:thm1_4}, we have
\begin{equation}
\begin{array}{r@{}l}
\begin{aligned}
\left< T_{\rm r}\, g, g \right> \leq \frac 1 2 \varepsilon_1 \int_\Omega 
\left( \left|\nabla u_{\rm r}\right|^2 +  \left|\nabla u_{\rm
e}\right|^2\right).
\end{aligned}
\end{array}
\end{equation}
Combined with \eqref{eq:thm1_2} and \eqref{eq:thm1_1}, we obtain
\begin{equation}\label{ineq:TrTe}
\begin{array}{r@{}l}
\begin{aligned}
 \left< T_{\rm r}\, g, g \right>  \leq  \varepsilon_1
\int_\Omega  \left|\nabla u_{\rm e}\right|^2 \leq
\frac{\varepsilon_1}{\varepsilon_2} \left< T_{\rm e}\, g, g \right>,
\end{aligned}
\end{array}
\end{equation}
that is,
\begin{equation}
 \left< (T_{\rm r}+T_{\rm c} ) g, g \right>  \leq
\max\left\{\frac{\varepsilon_1}{\varepsilon_2},~1 \right\} \left< (T_{\rm e} +
T_{\rm c} ) g, g \right>.
\end{equation}
 The remaining question is to
prove that there exists a constant $C>0$ independent of $g$, such that,
\begin{equation}\label{eq:thm1_8}
\left< T_{\rm e}\, g, g \right> \leq C \left< \left(T_{\rm r}+T_{\rm
c}\right)g, g \right>.
\end{equation}
According to \eqref{eq:thm1_3}, we have
\begin{equation}\label{eq:3.20}
\begin{array}{r@{}l}
\begin{aligned}
\left< T_{\rm e}\, g, g \right> \leq \frac 1 2 \varepsilon_2 \int_\Omega 
\left( \left|\nabla  u_{\rm e}\right|^2 +  \kappa^2 u_{\rm e}^2\right)
+ \frac 1 2 \varepsilon_2 \int_\Omega  \left( \left|\nabla u_{\rm
r}\right|^2 +  \kappa^2 u_{\rm r}^2\right),
\end{aligned}
\end{array}
\end{equation}
that is, 
\begin{equation}\label{eq:thm1_10}
\begin{array}{r@{}l}
\begin{aligned}
 \left< T_{\rm e}\, g, g \right> \leq
\frac{\varepsilon_2}{\varepsilon_1} \left< T_{\rm r}\, g, g \right>   + 
 {\varepsilon_2}\, \kappa^2 \int_\Omega   u_{\rm r}^2,
\end{aligned}
\end{array}
\end{equation}
obtained from \eqref{eq:thm1_1} and \eqref{eq:thm1_2}. Using Lemma 
\ref{lem:1}, we then obtain 
\begin{equation}
\left< T_{\rm e}\, g, g \right> \leq \frac{\varepsilon_2}{\varepsilon_1} \left<
T_{\rm r}\, g, g \right> +  \max\{1,\kappa^{2}\}  C_{\rm ie}^0\left< T_{\rm c}\, g, g \right>.
\end{equation}
This leads to 
\begin{equation}
\left< ( T_{\rm e}+T_{\rm c} ) g, g \right> \leq
\max\left\{\frac{\varepsilon_2}{\varepsilon_1}, 1+  \max\{1,\kappa^{2}\}  C_{\rm ie}^0
\right\}\left<(T_{\rm r}+T_{\rm c})\, g, g \right>.
\end{equation}
Therefore, \eqref{eq:thm1} holds true.
\end{proof}

\begin{remark}
Using the Green's formula and the trace theorem \eqref{eq:tracethm}, we can deduce another interior-exterior estimate as follows:
    \begin{equation}
    \begin{aligned}
\left< T_{\rm e}\, g, g \right> & = \varepsilon_2 
\int_\Omega \left( \left|\nabla u_{\rm e}\right|^2+ \kappa^2 u_{\rm
e}^2\right) \leq \varepsilon_2 \max\{1,\kappa^2\} \|u_{\rm e}\|_{H^1(\Omega)}^2 \\
& \leq \varepsilon_2 \max\{1,\kappa^2\} C_\Omega \|g\|_{H^{\frac12}(\Gamma)}^2 \leq \varepsilon_2 \max\{1,\kappa^2\} \widetilde C_\Omega \|u_{\rm c}\|_{H^1(\Omega^{\mathsf c})}^2\\
& \leq \widetilde C_\Omega \max\{\kappa^{-2},\kappa^2\} \varepsilon_2 
\int_{\Omega^{\mathsf c}} \left( \left|\nabla u_{\rm c}\right|^2+ \kappa^2 u_{\rm
c}^2\right) \\
& = \widetilde C_\Omega \max\{\kappa^{-2},\kappa^2\} \left< T_{\rm c}\, g, g \right>,
\end{aligned}
\end{equation}
which gives
\begin{equation}
\left< ( T_{\rm e}+T_{\rm c} ) g, g \right> \leq
( 1+  \widetilde C_\Omega \max\{\kappa^{-2},\kappa^2\}) 
\left<(T_{\rm r}+T_{\rm c})\, g, g \right>.
\end{equation}
Here $C_\Omega$ and $\widetilde C_\Omega$ are constants derived from the trace theorem. 
However, when $\kappa\rightarrow 0$ (the LPB model tends to the PCM), the constant $\max\{\kappa^{-2},\kappa^2\}\rightarrow \infty$, making the estimation inequality useless.
Moreover, in practical implicit solvent models, the constant $\kappa$ is usually far less than $1$. Therefore, we adopt the interior-exterior estimate in Theorem \ref{thm:1}, where $C_1$ is independent of $\kappa$ when $\kappa\leq 1$.
\end{remark}

\subsection{Convergence of Richardson iteration}\label{sect3.2}

We now prove rigorously the convergence of the ddLPB solver, i.e., the convergence of Richardson iteration \eqref{eq:iteration}. 
This convergence actually results from the spectral equivalence between $T_{\rm e} + T_{\rm c}$ and $T_{\rm r} + T_{\rm c}$.

We first recall some preliminary results on the spectrum of bounded linear operators in the following lemma.
\begin{proposition}[Properties on spectrum, \cite{reed1972methods}]
	Let $X$ be a Banach space and $L(X)$ be the spaces for bounded linear operators from $X$ into itself. Then :
	\begin{itemize}
		\item[(1)] If $A \in L(X)$ and $P$ is a polynomial, then we have
		\begin{equation}
		\sigma ( P(A) ) = \{ P(\lambda) : \lambda \in \sigma(A) \},
		\end{equation}
		where $\sigma(A)$ is the spectrum of the operator $A$.
		\item[(2)] (Spectral inclusion) If $H$ is a Hilbert space with the inner
product $\langle\cdot,\cdot\rangle_{H}$ and $T\in L(H)$, then 
\begin{equation}
	\sigma\big( T \big) \subset \overline{W(T)} \quad{\rm where} \quad W(T):=\left\{\langle T g,g \rangle_{H} :~ g\in H, ~\langle g,g \rangle_{H} = 1\right\}. \label{eq:numrange}
\end{equation}
\end{itemize}
\label{proppre}
\end{proposition}
\begin{proof}
	See \cite[Lemma 1 following Theorem \uppercase\expandafter{\romannumeral7}.1,
Theorem \uppercase\expandafter{\romannumeral6}.6]{reed1972methods} and
\cite[Theorem 1.2-1]{Numerical_Range}.
\end{proof}

The following theorem shows that the convergence of iteration \eqref{eq:iteration} can be deduced from the spectral equivalence of $T_{\rm e} + T_{\rm c}$ and $T_{\rm r} + T_{\rm
c}$.
\begin{theorem}
The spectrum of $\left(T_{\rm e} + T_{\rm c}\right)^{-1} (T_{\rm r} + T_{\rm
c})$ is bounded as follows
\begin{equation}
	\sigma\big( (T_{\rm e} + T_{\rm c})^{-1}(T_{\rm r} + T_{\rm c}) \big) \subset [C_1, C_2]. \label{eq:spectrum}
\end{equation}
where the constants $C_1$ and $C_2$ are defined in Theorem \ref{thm:1}. \label{thm31}
\end{theorem}
\begin{proof}

We first provide a new scalar product over $H^{\frac12}(\Gamma)$ that will be suitable for the proof. We notice that the bilinear forms $\langle T_{\rm e}\cdot,\cdot\rangle$ and $\langle T_{\rm c}\cdot,\cdot\rangle$ are symmetric and positive definite.
Indeed, for $g_1,~g_2\in H^{\frac12}(\Gamma)$, denote $u_{\rm e,1}$ and $u_{\rm e,2}$ the solutions of \eqref{eq:Te} with the boundary $g_1$ and $g_2$, respectively.
Multiplying $u_{\rm e,2}$ and $u_{\rm e,1}$ on both sides of \eqref{eq:Te} and integrating over $\Omega$, we have 
\begin{align}\label{eq:self-adjoint}
	\left<T_{\rm e}\,g_1, g_2\right> = \varepsilon_2 \int_{\Omega} (\nabla u_{\rm e,1} \cdot \nabla u_{\rm e,2} + \kappa^2 u_{\rm e,1}u_{\rm e,2}) = \left<g_1,T_{\rm e}\,g_2\right>.
\end{align}
When $g_1=g_2$, we have
\begin{equation}
\label{eq:H1e}
	\left< T_{\rm e}\, g_1, g_1 \right>= \varepsilon_2  \int_\Omega \left( \left|\nabla u_{\rm e,1}\right|^2+ \kappa^2 u_{\rm e,1}^2\right)\geq 0.
\end{equation}
When $\left< T_{\rm e}\, g_1, g_1 \right>=0$, the fact that $\kappa>0$
yields $u_{\rm e,1}=0$, which implies $g_1 = u_{\rm e,1}|_{\Gamma}
= 0$. 
Thus, $\langle T_{\rm e}\cdot,\cdot\rangle$ is symmetric positive-definite on $H^{\frac12}(\Gamma)$. A similar argument shows that $\langle T_{\rm c}\cdot,\cdot\rangle$ is also symmetric positive-definite on $H^{\frac12}(\Gamma)$.
%
This allows us to define an inner product on $H^{\frac 1 2}(\Gamma)$ as follows
\begin{equation}
\langle g_1,g_2 \rangle_{\rm ec} := \langle (T_{\rm e}+T_{\rm
c})g_1,g_2\rangle \quad \forall g_1,~g_2 \in H^{\frac 1 2}(\Gamma),
\end{equation}
and an induced norm by 
\begin{equation}
\lVert g \Vert_{\rm ec} := \sqrt{ \left<g\, ,g\right>_{\rm ec}}.
\end{equation}

Then we show that this norm $\| \cdot \|_{\rm ec}$ is equivalent to the original one $\|\cdot\|_{H^{\frac12}(\Gamma)}$. 
For any $g\in H^{\frac12}(\Gamma)$, denote $u_{\rm e}$ and $u_{\rm c}$ the solutions of \eqref{eq:Te} and \eqref{eq:Tc}, respectively. A similar argument as \eqref{eq:H1e} gives
\begin{equation}
\label{eq:H1ec}
	\|g\|_{\rm ec}^2= \varepsilon_2  \int_\Omega \left( \left|\nabla u_{\rm e}\right|^2+ \kappa^2 u_{\rm e}^2\right) + \varepsilon_2  \int_{\Omega^{\mathsf c}} \left( \left|\nabla u_{\rm c}\right|^2+ \kappa^2 u_{\rm c}^2\right) \geq C\|u_{\rm e}\|^2_{H^1(\Omega)},
\end{equation}
where $C$ depends on $\varepsilon_2$ and $\kappa$. This together with the fact that the trace operator
$\gamma_0:H^1(\Omega)\rightarrow H^{\frac 1 2}(\Gamma)$ is bounded \cite[Theorem
2.6.8]{sauter2010BEM}, gives
\begin{equation}\label{eq:equ_norm1}
	\|g\|^2_{H^{\frac12}(\Gamma)} \leq C \|u_{\rm e}\|^2_{H^1(\Omega)} \leq C\|g\|_{\rm ec}^2.
\end{equation}
According to \cite[Theorem 2.7.7,~Theorem 2.10.4,~Theorem 2.10.7]{sauter2010BEM}, 
we have
\begin{equation}\label{eq:3.23}
\| (T_{\rm e} + T_{\rm c})\, g \|_{H^{-\frac 1 2}(\Gamma)}
\leq C(\| u_{\rm e}\|_{H^1(\Omega)} + \| u_{\rm c}\|_{H^1(\Omega^{\mathsf c})})
\leq C\|g\|_{H^{\frac12}(\Gamma)},
\end{equation}
implying
\begin{equation}\label{eq:equ_norm2}
	\|g\|_{\rm ec}^2 = \left< (T_{\rm e}+T_{\rm c})\, g, g \right> \leq \| (T_{\rm e} + T_{\rm c})\, g \|_{H^{-\frac 1 2}(\Gamma)} \|g\|_{H^{\frac12}(\Gamma)} \leq C \|g\|_{H^{\frac12}(\Gamma)}^2.
\end{equation}
Combining \eqref{eq:equ_norm1} and \eqref{eq:equ_norm2}, we conclude that the norms $\|\cdot\|_{\rm ec}$ and $\|\cdot\|_{H^{\frac12}(\Gamma)}$ are equivalent, hence $\big( H^{\frac 1 2}(\Gamma), ~\left<\cdot,\cdot\right>_{\rm ec} \big)$ is a Hilbert space.

Let us now consider the operator $(T_{\rm e} + T_{\rm c})^{-1}(T_{\rm r} + T_{\rm c}) : \big(
H^{\frac 1 2}(\Gamma), ~\langle\cdot,\cdot\rangle_{\rm ec} \big) \rightarrow
\big( H^{\frac 1 2}(\Gamma), ~\langle\cdot,\cdot\rangle_{\rm ec} \big)$. 
We derive from \eqref{eq:inv_Ska} that $(T_{\rm e} + T_{\rm c})^{-1}= \varepsilon_2^{-1}\mathcal{S}_{\kappa}$, hence the operator $(T_{\rm e} + T_{\rm c})^{-1} $ is bounded  from $H^{-\frac12}(\Gamma) $ into $ H^{\frac 1 2}(\Gamma) $ \cite[Theorem 3.1.16]{sauter2010BEM}, combining with the boundedness of $T_{\rm r}+T_{\rm c}$ from \eqref{eq:3.23}, we know that $(T_{\rm e} + T_{\rm c})^{-1}(T_{\rm r} + T_{\rm c})$ is a bounded operator on $H^{\frac12}(\Gamma)$. 
The equivalence of $\|\cdot\|_{\rm ec}$ and $\|\cdot\|_{H^{\frac12}(\Gamma)}$ yields that it is also bounded on $\big( H^{\frac 1 2}(\Gamma),
~\langle\cdot,\cdot\rangle_{\rm ec} \big)$

Since $ \left< (T_{\rm e} + T_{\rm c})^{-1}(T_{\rm r} + T_{\rm c}) \,g ,g \right>_{\rm ec} = \left< \left(T_{\rm r}+T_{\rm c}\right)g,g\right>$ and $\|g\|_{\rm ec}^2 = \left<\left(T_{\rm e}+T_{\rm c}\right)g,g\right>$, we deduce from \eqref{eq:thm1} and \eqref{eq:numrange} that
\begin{equation}
	\sigma\big( (T_{\rm e} + T_{\rm c})^{-1}(T_{\rm r} + T_{\rm c}) \big) \subset \overline{W\big( (T_{\rm e} + T_{\rm c})^{-1}(T_{\rm r} + T_{\rm c}) \big)}
\subset [C_1,C_2],
\end{equation}
where we take the Hilbert space $H = \big(
H^{\frac 1 2}(\Gamma), ~\langle\cdot,\cdot\rangle_{\rm ec} \big)$ and $T= (T_{\rm e} + T_{\rm c})^{-1}(T_{\rm r} + T_{\rm c})$ in the second statement of Proposition \ref{proppre}.
\end{proof}

\begin{remark}
    The result in Theorem \ref{thm31} can be directly obtained if the spectrum of $(T_{\rm e} + T_{\rm c})^{-1}(T_{\rm r} + T_{\rm c})$ only has eigenvalues. However, the spectrum of $(T_{\rm e} + T_{\rm c})^{-1}(T_{\rm r} + T_{\rm c})$ can not be reduced to point spectrum and it is shown in Appendix \ref{append:B} that the continuous spectrum does exist for a simple sphere. 
    \end{remark}

\begin{theorem}\label{thm3}
    The Richardson iteration \eqref{eq:iteration} converges for $0 < \alpha < 2/{C_2}$ where $C_2=\max\left\{1,\frac{\varepsilon_1}{ \varepsilon_2}\right\}$.
    Furthermore, we have
    \begin{equation}\label{eq:3.38}
        \|g^{k+1}-g^*\|_{\rm ec} \leq \max\{|1-\alpha C_2|,~|1-\alpha C_1|\} \|g^k-g^*\|_{\rm ec},
    \end{equation}
    where $g^*$ is the limit of $g^k$, i.e., the exact Dirichlet boundary condition.
\end{theorem}
\begin{proof}
Using the first statement of Proposition \ref{proppre} and Theorem \ref{thm31} yields 
\begin{equation}
\label{eq:spe}
\rho\left(I - \alpha\, {\varepsilon_2^{-1}} \mathcal{S}_{\kappa} \left(T_{\rm r}
+ T_{\rm c}\right) \right) \leq \max\left\{|1-\alpha C_2|,~|1-\alpha C_1|\right\}.
\end{equation}
where $\rho(\cdot)$ is the spectral radius.
Thus, the convergence of the iteration \eqref{eq:iteration} is ensured if 
\begin{equation}
\max\{|1-\alpha C_2|,~|1-\alpha C_1|\}< 1,
\end{equation}
i.e.,  $0 < \alpha < 2/{C_2}$.

From the proof of Theorem \ref{thm31}, the operator $I - \alpha\,(T_{\rm e} + T_{\rm c})^{-1}(T_{\rm r} + T_{\rm c})$ is bounded and symmetric in the Hilbert space $\big(
H^{\frac 1 2}(\Gamma), ~\langle\cdot,\cdot\rangle_{\rm ec} \big)$.
According to the fact that the norm of a bounded symmetric operator equals its spectral radius (see for example \cite[Section 31.1]{lax2014}), we then have
\begin{equation}
    \rho(I - \alpha\,(T_{\rm e} + T_{\rm c})^{-1}(T_{\rm r} + T_{\rm c})) = \|I - \alpha\,(T_{\rm e} + T_{\rm c})^{-1}(T_{\rm r} + T_{\rm c})\|_{\rm ec},
\end{equation}
which leads to \eqref{eq:3.38}.
\end{proof}

Theorem \ref{thm3} claims that if $0<\alpha < 2/C_2$ where
$C_2$ is given by \eqref{eq:C12}, then \eqref{eq:iteration} is convergent. 
The convergent rate is determined by $\rho\left(I - \alpha\, {\varepsilon_2^{-1}} \mathcal{S}_{\kappa} \left(T_{\rm r} + T_{\rm c}\right) \right)$, thus by the upper bound and lower bound of the essential spectrum of ${\varepsilon_2^{-1}} \mathcal{S}_{\kappa} \left(T_{\rm r} + T_{\rm c}\right)$ and $\alpha$.
If $C_1$ and $C_2$ defined in \eqref{eq:C12} are sharp evaluations of the  lower bound and upper bound of the essential spectrum of $(T_{\rm e} + T_{\rm c})^{-1}(T_{\rm r}+T_{\rm c})$, then we can choose an optimal value $\alpha_{\rm op}$ for $\alpha$ to speed up the convergence of the Richardson iteration. Specifically, if we set 
\begin{equation}\label{eq:alphaop}
\alpha = \alpha_{\rm op} := \frac{2}{C_1+C_2},
\end{equation}
the spectral radius of iteration  \eqref{eq:iteration} satisfies
\begin{equation}
\rho\left(I - \alpha_{\rm op} \left(T_{\rm e}+T_{\rm c}\right)^{-1} \left(T_{\rm r} +
T_{\rm c}\right) \right)
\leq \frac{C_2-C_1}{C_2+C_1} < 1.
\end{equation}


Let us remind that $C_{\rm ie}^0$ stands for the interior-exterior constant that appears in the definition of $C_1$ in \eqref{eq:C12}, for realistic solvent models, we thus have $\varepsilon_1 =1$, $\varepsilon_2> \varepsilon_1$, and $\kappa < 1$. By \eqref{eq:C12}, we obtain
\begin{equation}
C_2 = 1, \quad
\alpha_{\rm op} = 2 - \frac{2}{\max\{\varepsilon_2,~C_{\rm ie}^0+1\} + 1}.
\end{equation}
The spectral equivalence constants $C_1$ and
$C_2$ do not depend on $\kappa$. 
This means that as $\kappa \rightarrow 0$, the convergence can still be ensured. 
In fact, if we use the constant $C_{\rm ie}^1$, we can obtain another interior-exterior estimate like Lemma \ref{lem:1}:
\begin{equation}
\|u_{\rm r}\|_{H^1(\Omega)}^2  \leq C_{\rm ie}^1 \, \varepsilon_2^{-1} \max\{1,\kappa^{-2}\} 
\left< T_{\rm c}\, g, g \right>,
\end{equation}
which gives 
\begin{equation}\label{ineq:Teur}
    \left< T_{\rm e}\, g, g \right> \leq \varepsilon_2 \max\{1,\kappa^2\} \left\| u_r\right\|_{H^1(\Omega)}^2\leq C_{\rm ie}^1 \, \max\{\kappa^2,\kappa^{-2}\} \left< T_{\rm c}\, g, g \right>
\end{equation}
and then
\begin{equation}
    C_1 =  \frac{1}{1+\max\{\kappa^{2},\kappa^{-2}\} C_{\rm ie}^1}.
\end{equation}
Here the first inequality in \eqref{ineq:Teur} is derived from \eqref{eq:3.20}. 
In this case, $C_1\rightarrow 0$ as $\kappa\rightarrow 0$ (the LPB model becomes the PCM model), which is however not good because the optimal spectral radius $\rho$ tends to one.

\section{Numerical results}\label{sec:num}
In this section, we present numerical experiments to verify the convergence analysis of the generalized ddLPB method and find the optimal step size $\alpha$ in the Richardson iteration.

For simplicity, we consider the VDW cavity as $\Omega$, which can be decomposed into a group of overlapping balls:
$\Omega = \bigcup_{j = 1}^M \Omega_j,$
where $M$ is the number of atoms and $\Omega_j \in \mathbb{R}^3$ represents the $j$th atomic ball. We use the overlapping domain decomposition solvers presented in \cite{quan2019domain} to solve the internal Laplace and HSP equations in Step 2 and 3 of Algorithm \ref{algo}. Specifically, for the Laplace equation of $\psi_{\rm r}$ and the HSP equation of $\psi_{\rm e}$, we solve a group of coupled Laplace sub-equations and coupled HSP sub-equations, each defined in a ball. Since the Laplace or HSP subproblem in a ball has an explicit solution formula, the Laplace and HSP solver in $\Omega$ is efficient and scales linearly. More details on the interior solver can be found in \cite{quan2019domain}.

The solutions of \eqref{eq:dd_laplace} and \eqref{eq:dd_screen} on each sphere are approximated by a linear combination of spherical harmonics ${Y_{\ell}^m}$ with $0\leq \ell \leq \ell_{\max}$ and $-\ell\leq m\leq \ell$. At each external iteration, we need to solve the following linear system:
\begin{equation}
\left\{
\begin{aligned}
A X_{\rm r} &= G_{X} + G_0, \\[1ex]
B X_{\rm e} &= G_{X},
\end{aligned}
\right.
\end{equation}
where $X_{\rm r}$, $X_{\rm e}$, and $G_X$ are the vectors of the coefficients of spherical harmonics for $\psi_{\rm r}$, $\psi_{\rm e}$, and $g$, respectively.
In the generalized ddLPB method, we numerically update the Dirichlet boundary condition \eqref{eq:update} as follows:
\begin{equation}
\label{eq:GX}
G_{X}^{(k)} = (1-\alpha) G_{X}^{(k-1)} + \alpha( F_0 - C_1 X_{\rm r}^{(k)} -
C_2 X_{\rm e}^{(k)}),
\end{equation}
which is slightly different from the original ddLPB method in \cite{quan2019domain}.

By default, we take the dielectric permittivity in the solute cavity to be in vacuum, that is, $\varepsilon_1=1$. Moreover, we set the Debye--H\"uckel screen constant to $\kappa=0.1040$ \AA$^{-1}$ for an ionic strength of $I=0.1$ molar.
The atomic centers, charges, and VDW radii are obtained from the PDB files \cite{berman2000protein} and the PDB2PQR package \cite{dolinsky2004pdb2pqr,dolinsky2007pdb2pqr} with the setting of PEOEPB force field. 
Our convergence criterion is based on the relative error of energy, which is defined as 
\begin{equation}
{\tt Err}_k \coloneqq  \frac{ |E^{(k)} - E^{(k-1)}|}{|E^{(k-1)}| } < {\tt tol} = 10^{-4},
\end{equation} 
where $E^{(k)}$ and $E^{(k-1)}$ are the electrostatic solvation energies (see \cite{quan2019domain} for the computation) at iteration $k$ and $k-1$, respectively, and ${\tt tol}$ is a predefined tolerance. We also set a maximum iteration number of $k_{\rm max} = 60$. An iteration is considered ``convergent'' if the relative error of energy satisfies the convergence criterion within $k_{\rm max}$ steps.





\subsection{Convergence test}\label{sect:conv}
As a first step, we test the convergence of the generalized ddLPB algorithm for two simple cases. One is a single sphere and the other is two spheres with same radius $R=1$. 
In the case of two spheres, the distance between their centers is $0.01R$. 
To run the solver, we set the maximum degree of spherical harmonics to $\ell_{\rm max} = 7$ and the number of Lebedev points to $N_{\rm leb} = 86$ (for more details, please refer to \cite{quan2019domain}).

In Figure \ref{fig:1}, we plot the iteration number against different values of $\alpha\in \{0.1,0.2,\ldots,2.0\}$, corresponding to $\varepsilon_2=2$, $1$, and $0.5$.
Note that we only plot the convergent cases, i.e., the iterations that stopped before reaching the maximum iteration number $k_{\rm max}=60$.

According to Theorem \ref{thm:1}, the ddLPB iteration \eqref{eq:iteration} converges for $0<\alpha<2/{C_2}$, where $C_2=\max\{1,\varepsilon_1/\varepsilon_2\}$.
Hence, when $\varepsilon_2$ equals $1$ or $2$, the ddLPB algorithm will converge for $0<\alpha<2$. Similarly, when $\varepsilon_2$ is $0.5$, the ddLPB algorithm will converge for $0<\alpha<1$.
These results are also observed in Figure \ref{fig:1}.

\begin{figure}[htb!]
\centering
	\subfigure[1 sphere]{
		\includegraphics[width=0.45\textwidth]{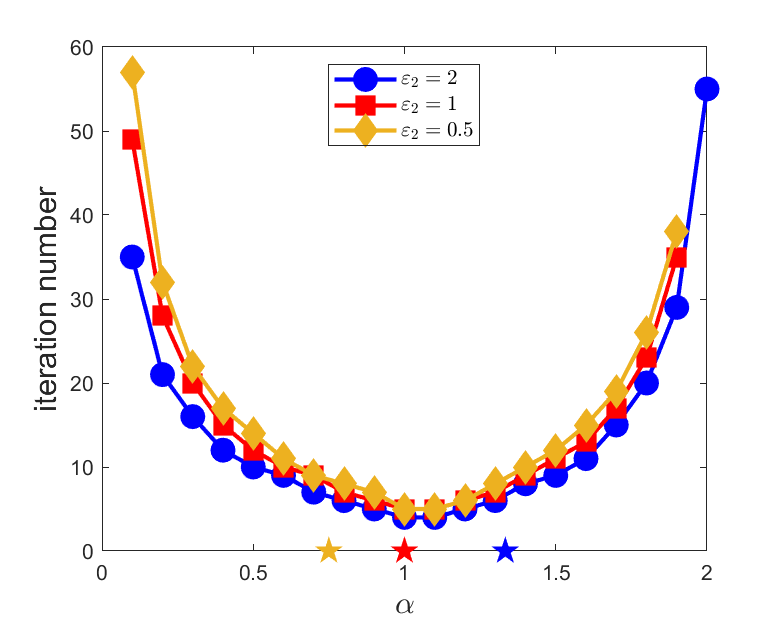}
	}
	\subfigure[2 spheres]{
		\includegraphics[width=0.45\textwidth]{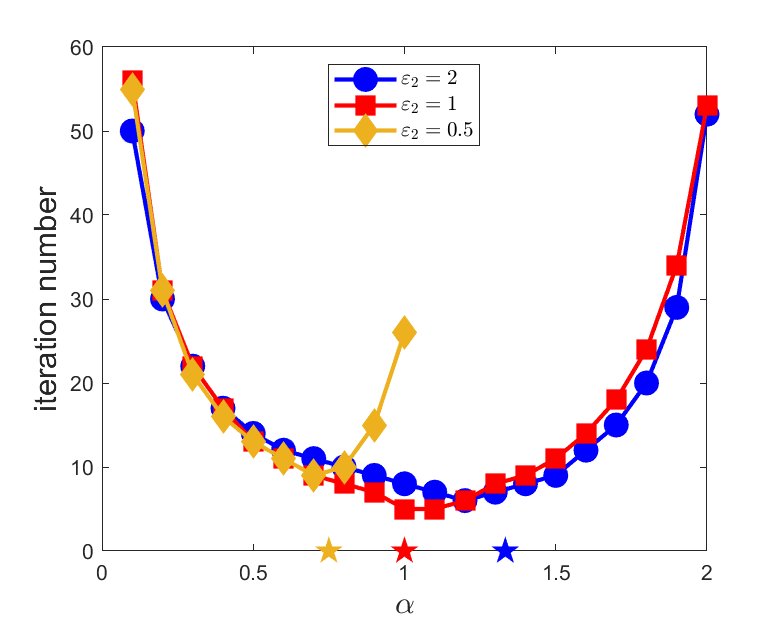}
	}
\caption{Iteration number $N_{\rm ite}(\alpha)$ v.s. the relaxation parameter $\alpha$ for 1 sphere and 2 spheres  ($\varepsilon_1=1,~R=1,~\kappa=1$\AA$^{-1}$).
The ``optimal'' parameter $\bar\alpha_{\rm op}$ computed by \eqref{eq:alphaop2} is the abscissa of star with different colors corresponding different $\varepsilon_2$.}
\label{fig:1}
\end{figure}

In the case of one sphere, when $\varepsilon_2=0.5$, We observe that ddLPB converge not only for $\alpha\in(0,1)$ but also for $\alpha\in(1,2)$. The reason for this may be that $C_2$ is a upper bound in Theorem \ref{thm31} and not sharp enough for this case. However, we also observe from the case of two spheres that the converge range of $\alpha$ becomes $(0,1)$ once there are two spheres close to each other (with distance between centers to be  $0.01R$). Note that in this two cases, the domains have very little difference.

\subsection{Optimal relaxation parameters}
We consider to choose $\alpha\in \{0.1,0.2,\ldots,2.0\}$.
When the iteration number reaches the minimum, the corresponding $\alpha$ is denoted by $\tilde\alpha_{\rm op}$ given by
\begin{equation}\label{eq:alphaop_tilde}
    \tilde \alpha_{\rm op} \coloneqq  \mathop{\rm argmin}\limits_{\alpha \in\{ 0.1,0.2,\ldots,2.0\}} N_{\rm ite}(\alpha),
\end{equation}
where $N_{\rm ite}(\alpha)$ is the iteration number for a fixed $\alpha$ to reach the convergence tolerance.
According to \eqref{eq:alphaop}, we obtain
\begin{equation}
    \alpha_{\rm op} = \frac{2}{\min\{\varepsilon_1/\varepsilon_2,1/(1+C_{\rm ie}^0)\} + \max\{1,\varepsilon_1/\varepsilon_2\}}. 
\end{equation}
As estimating $C_{\rm ie}^0$ is difficult, we ignore the term $1/(1+C_{\rm ie}^0)$ in the above equation. Thus, we have
\begin{equation}\label{eq:alphaop2}
  \alpha_{\rm op} \approx \bar \alpha_{\rm op}\coloneqq \frac{2}{\varepsilon_1/\varepsilon_2 + \max\{1,\varepsilon_1/\varepsilon_2\}}
    = \frac43,~1,~\frac34 \quad \mbox{for}~ \varepsilon_2=2,~1,~0.5,~\mbox{respectively.}
\end{equation} 
We run the ddLPB algorithm for 1 sphere and 2 spheres with the same discretization settings as in Subsection \ref{sect:conv}.
From the right-hand side of Figure \ref{fig:1}, we observe that for 2 spheres, 
$
     \tilde \alpha_{\rm op}
    = 1.3,~1,~ 0.7 \mbox{~for}~ \varepsilon_2=2,~1,~0.5,~\mbox{respectively.}
$
We also notice that $\bar\alpha_{\rm op}$ is close to $\tilde\alpha_{\rm op}$ for different values of $\varepsilon_2$, which verifies our analysis. Therefore, in practical implementations of the ddLPB algorithm, it is advisable to use $\bar\alpha_{\rm op}$ defined in \eqref{eq:alphaop2} as the guess for the optimal relaxation parameter, especially for not large $\varepsilon_2$.

We now investigate the relationship between $\tilde\alpha_{\rm op}$ and $\varepsilon_2\geq 1$. The results are presented in Figure \ref{fig:e2_op}. In fact, when $\varepsilon_2\geq 1$, \eqref{eq:alphaop2} simplifies to:
\begin{equation}\label{eq:alphaop3}
\bar\alpha_{\rm op}= \frac{2}{\varepsilon_2^{-1} + 1},
\end{equation}
which is plotted as a reference curve in Figure \ref{fig:e2_op} 
(the estimation of $C_{\rm ie}^0$ is not considered).

Despite the speciality of 1 sphere case, 
it can be observed that $\tilde\alpha_{\rm op}$ initially increases as $\varepsilon_2$ increases. However, as $\varepsilon_2$ becomes sufficiently large, $\tilde\alpha_{\rm op}$ reaches a plateau and remains relatively constant. The profile of $\tilde\alpha_{\rm op}$ is quite similar to that of $\bar\alpha_{\rm op}$ in \eqref{eq:alphaop3}, except that $\tilde\alpha_{\rm op}$ is smaller. The reason for this may be that $C_1$ is only a lower bound in \eqref{eq:thm1} and is not necessarily the essential infimum. Consequently, the numerical optimal relaxation parameter $\tilde \alpha_{\rm op}$ may be smaller than $\bar\alpha_{\rm op}$ due to the possibility of the essential infimum being greater than $C_1$.

\begin{figure}[h!]
\centering
\includegraphics[width=0.6\textwidth]{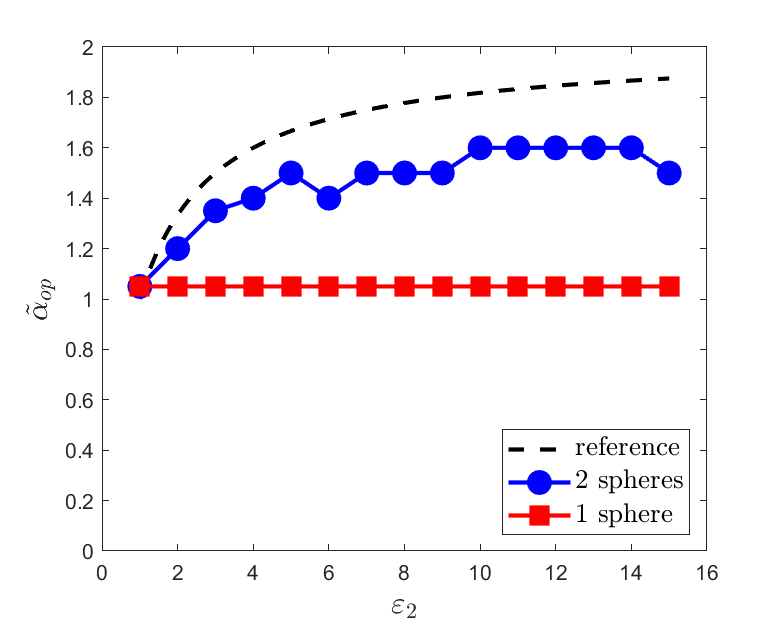}
\caption{ $\tilde\alpha_{\rm op}$ defined in \eqref{eq:alphaop_tilde} v.s. $\varepsilon_2$ for 1 sphere and 2 spheres ($\varepsilon_1=1,~R=1,~\kappa=1$\AA$^{-1}$), and $\bar\alpha_{\rm op}$ in \eqref{eq:alphaop3} is the black dash line. }
\label{fig:e2_op}
\end{figure}

We now investigate how $R$ and $\kappa$ influence the convergence. 
We increase $\alpha$ from $0.1$ by $0.1$. The largest $\alpha$ that endures the convergence is denoted by $\tilde{\alpha}_{\rm max}$. 
We compute $\tilde{\alpha}_{\rm max}$ and $\tilde{\alpha}_{\rm op}$ for different $R$ and $\kappa$ for 2 spheres and the results are in Figure \ref{fig:conv2}. 
We see that as $R$ and $\kappa$ double, $\tilde{\alpha}_{\rm max}$ shows a linear growth trend, while $\tilde{\alpha}_{\rm op}$ gradually increases, getting closer and closer to $\tilde{\alpha}_{\rm max}$.
Since the Richardson iteration operator is $I - \alpha \left(T_{\rm e} + T_{\rm c}\right)^{-1}\left(T_{\rm r} + T_{\rm c}\right)$, we learn that with greater $R$ or greater $\kappa$, the essential upper bound of $\sigma(\left(T_{\rm e} + T_{\rm c}\right)^{-1}\left(T_{\rm r} + T_{\rm c}\right) )$ is lower and its essential lower bound gets larger. 
This shows that the estimate of the spectrum of Richardson iterations might be improved in future work.

However, as we will demonstrate in the next subsections, our theoretical results can also offer some practical benefits.

\begin{figure}[htb!]
\centering
\subfigure{
    \includegraphics[width=0.45\textwidth]{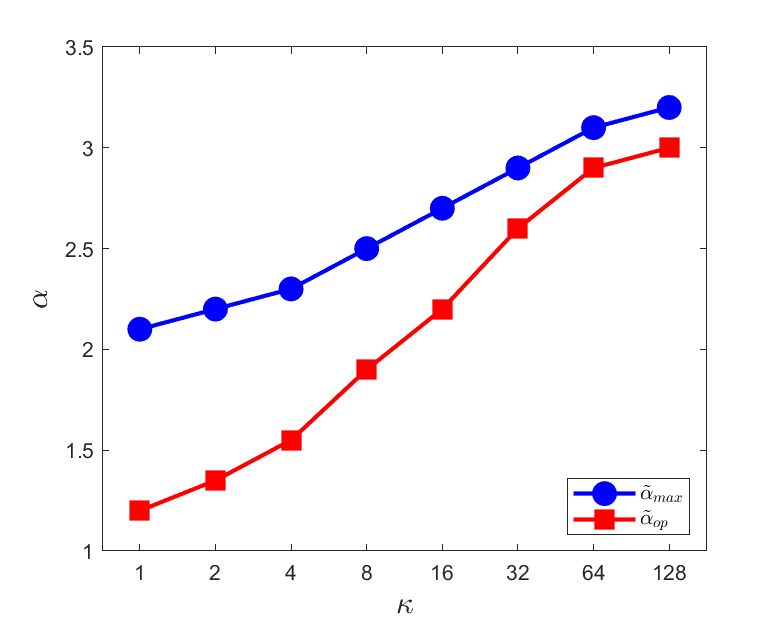}
}
\subfigure{
    \includegraphics[width=0.45\textwidth]{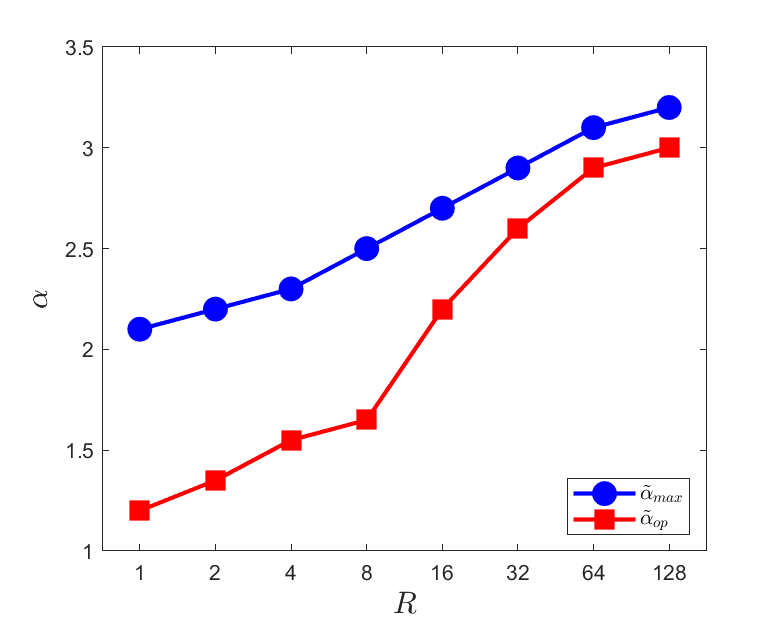}
}
\caption{Left: $\tilde{\alpha}_{\rm op},~\tilde{\alpha}_{\rm max}$ v.s. $R$ for 2 spheres ($\varepsilon_1=1,~\varepsilon_2=2,~\kappa=1$\AA$^{-1}$). Right: $\tilde{\alpha}_{\rm op},~\tilde{\alpha}_{\rm max}$ v.s. $\kappa$ for 2 spheres ($\varepsilon_1=1,~\varepsilon_2=2,~R=1$).}
\label{fig:conv2}
\end{figure}

\subsection{Benzene and caffeine in water solvent}
We now examine the water solvent models for benzene and caffeine with the same discretization settings as in Subsection \ref{sect:conv}. In this scenario, the dielectric permittivity of water is $\varepsilon_2=78.54$ at room temperature $T=298.15$K ($25^{\circ}$C).
Table \ref{table:1} shows the results of $N_{\rm ite}(\alpha)$ for $\alpha=0.1,0.2,\ldots, 2.0$.
We observe that the iteration converges for $\alpha\in(0,2)$ for these molecules, which aligns with our theory.
From this table, $\tilde \alpha_{\rm op}$ is around $1.8$ and $1.5$ respectively for benzene and caffeine. 

Furthermore, the computed energies for benzene and caffeine using different relaxation parameters are nearly identical. Specifically, for benzene, all energy values fall within the range of (-0.2068,-0.2063), while for caffeine, the range is (-2.1989,-2.1952). While the energy values may vary slightly depending on the choice of $\alpha$, the number of iterations required to converge may differ. Therefore, selecting an appropriate $\alpha$ can significantly accelerate the algorithm.

\begin{table}[!ht]
\renewcommand\arraystretch{1.6}
\centering
\scalebox{0.88}{
	\begin{tabular}{|m{44pt}<{\centering}|m{13pt}<{\centering}|m{13pt}<{\centering}|m{13pt}<{\centering}|m{13pt}<{\centering}|m{13pt}<{\centering}|m{13pt}<{\centering}|m{13pt}<{\centering}|m{13pt}<{\centering}|m{13pt}<{\centering}|m{13pt}<{\centering}|m{90pt}<{\centering}|}
	\hline
	\multirow{2}*{Molecules}  & \multicolumn{10}{c|}{Number of iterations with different $\alpha$} & Energy range
	\\ \cline{2-11}
	& $0.1$ & $0.2$ & $0.3$ & $0.4$ & $0.5$ & $0.6$ & $0.7$ & $0.8$ & $0.9$ & $1.0$ & (kJ/mol)
	\\ \hline 
	benzene (12) & $47$ & $32$ & $24$ & $20$ & $17$ & $15$ & $13$ & $12$ & $11$ & $10$ & $(-0.2068,-0.2063)$
	\\ \hline
	caffeine (24) & $30$ & $22$ & $17$ & $14$ & $12$ & $11$ & $10$ & $9$ & $8$ & $7$ & $(-2.1989,-2.1955)$
	\\ \hline
	\multirow{2}*{Molecules}  & \multicolumn{10}{c|}{Number of iterations with different $\alpha$} & Energy range
	\\ \cline{2-11}
	& $1.1$ & $1.2$ & $1.3$ & $1.4$ & $1.5$ & $1.6$ & $1.7$ & $1.8$ & $1.9$ & $2.0$ & (kJ/mol)
	\\ \hline
	benzene (12) & $9$ & $9$ & $8$ & $7$ & $7$ & $7$ & $6$ & \red $5$ & \red $5$ & $7$ & $(-0.2063,-0.2063)$
	\\ \hline
	caffeine (24) & $7$ & $7$ & $6$ & $6$ & \red$3$ & $5$ & $6$ & $7$ & $8$ & $10$ & $(-2.1964,-2.1952)$
	\\ \hline
	\end{tabular}}
\caption{Number of iterations and energy ranges for small molecules in water ($\varepsilon_1=1,~\varepsilon_2=78.54,~\kappa=0.104$ \AA$^{-1}$). The numbers of atoms are indicated in parentheses next to the molecule names.}
\label{table:1}
\end{table}

\subsection{Protein molecules}
We test several protein molecules using the PDB codes 1a3y, 2olx, 1yjo, 1etn, ala25, and 1bbl. To use the ddLPB solver, we set $\ell_{\rm max} = 7$ and $N_{\rm leb} = 86$ for 1a3y, 2olx, 1yjo, 1etn, and ala25. However, due to the high computational cost in Matlab, we set $\ell_{\rm max} = 5$ and $N_{\rm leb} = 50$ for 1bbl, which has 576 atoms.

Using the same settings as Figure \ref{fig:1}, we plot the iteration number versus different values of $\alpha$ for $\varepsilon_2 = 2,~1,~0.5$ in Figure \ref{fig:2}. Specifically, we test $\alpha$ values in the range ${0.1, 0.2, \ldots, 2.0}$. As shown in the figure, when $\varepsilon_2 = 1,~2$, the ddLPB converges for $\alpha\in(0,2)$, while for $\varepsilon_2 = 0.5$, it converges for $\alpha\in(0,1)$, which is consistent with Theorem \ref{thm:1}. Furthermore, Figure \ref{fig:2} demonstrates that $\bar\alpha_{\rm op}$ in \eqref{eq:alphaop2} is similar to $\tilde\alpha_{\rm op}$.

\begin{figure}[htb!]
\centering
	\subfigure[1ay3 ($25$ atoms)]{
		\includegraphics[width=0.45\textwidth]{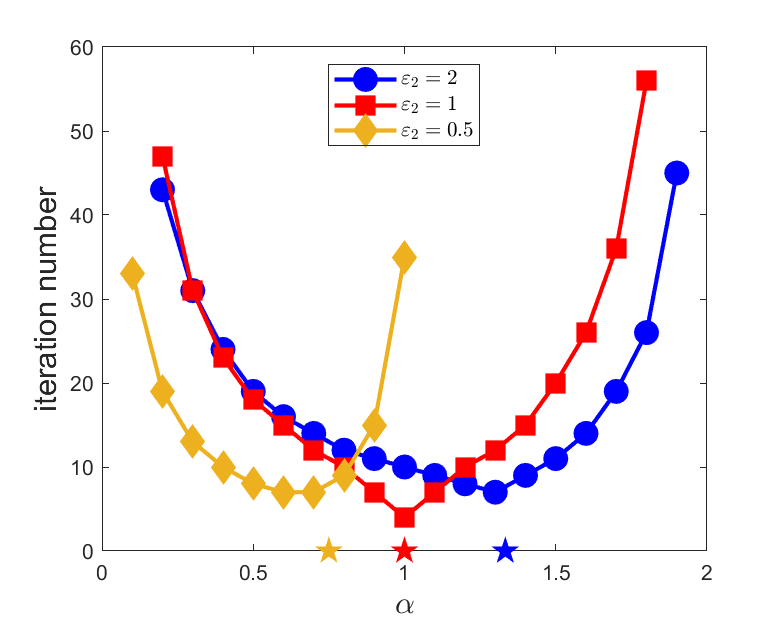}
	}
	\subfigure[2olx ($65$ atoms)]{
		\includegraphics[width=0.45\textwidth]{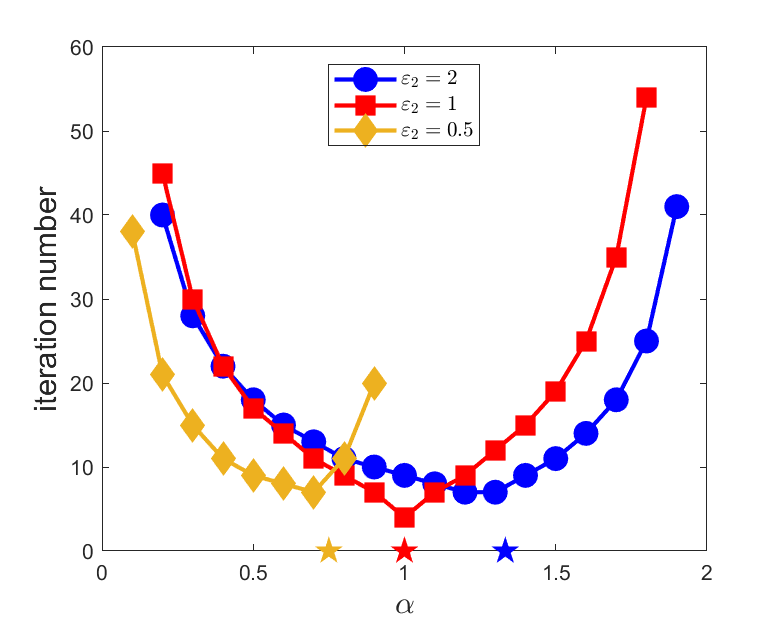}
	}
	\subfigure[1yjo ($121$ atoms)]{
		\includegraphics[width=0.45\textwidth]{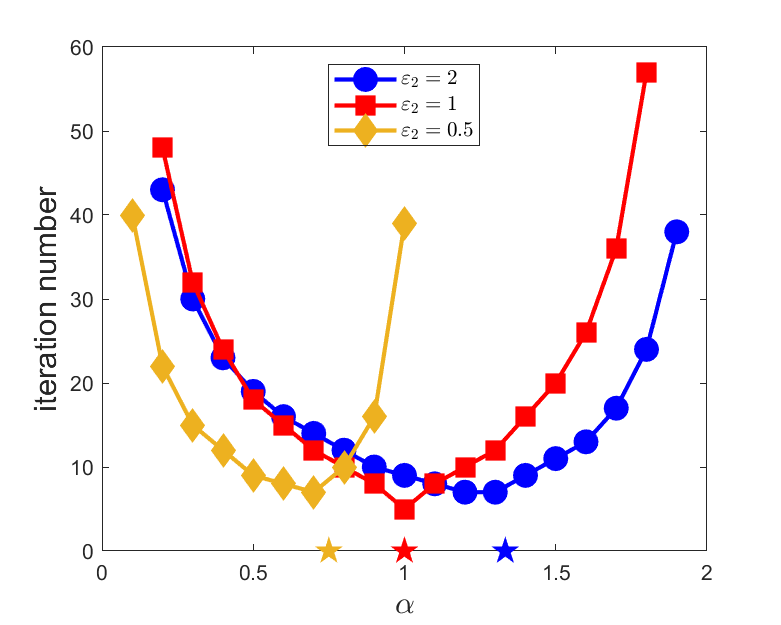}
	}
	\subfigure[1etn ($180$ atoms)]{
		\includegraphics[width=0.45\textwidth]{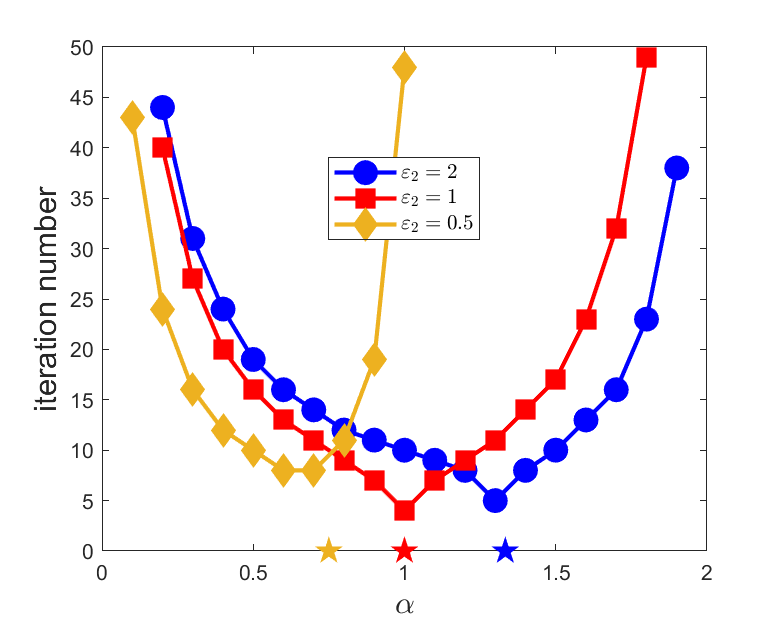}
	}
	\subfigure[ala25 ($259$ atoms)]{
		\includegraphics[width=0.45\textwidth]{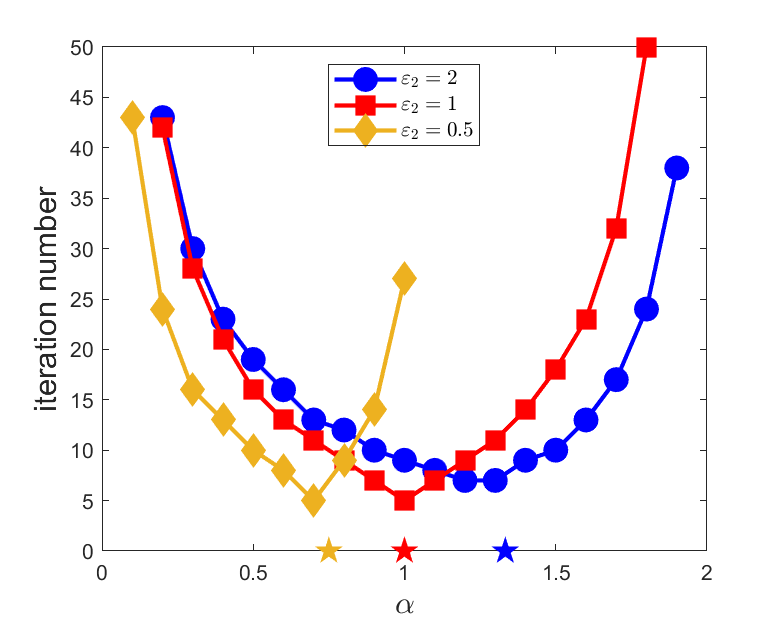}
	}
	\subfigure[1bbl ($576$ atoms)]{
		\includegraphics[width=0.45\textwidth]{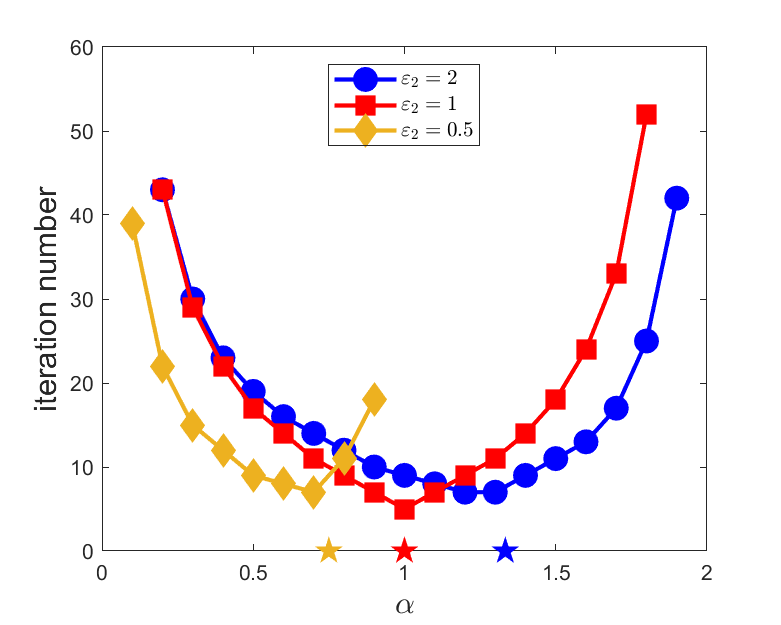}
	}
\caption{Iteration number $N_{\rm ite}(\alpha)$ v.s. the relaxation parameter $\alpha$ for protein molecules ($\varepsilon_1=1,~\kappa=0.104$\AA$^{-1}$).
The optimal parameter $\bar\alpha_{\rm op}$ in \eqref{eq:alphaop2} is the abscissa of star with different colors corresponding different $\varepsilon_2$.}
\label{fig:2}
\end{figure}

In this study, we evaluate the water solvent models with a dielectric constant $\varepsilon_2=78.54$ for the protein molecules under investigation. The results of $N_{\rm ite}(\alpha)$ for $\alpha$ values ranging from 0.1 to 2.0 are presented in Table \ref{table:2}. We observe that for a given protein molecule, the computed energies with different $\alpha$ values are similar, irrespective of the number of iterations. Our findings suggest that $\alpha = 1.7$ consistently provided optimal acceleration.

Compared to the original ddLPB ($\alpha=1$), the use of $\alpha=1.7$ in the generalized ddLPB method reduces the number of iterations, resulting in cost savings. That is to say, the use of optimized $\alpha$ value can improve the efficiency of the generalized ddLPB method in protein simulations.

\setlength{\tabcolsep}{2mm}{
\begin{table}[!ht]
\renewcommand\arraystretch{1.6}
\centering
\scalebox{0.88}{
	\begin{tabular}{|m{56pt}<{\centering}|m{13pt}<{\centering}|m{13pt}<{\centering}|m{13pt}<{\centering}|m{13pt}<{\centering}|m{13pt}<{\centering}|m{13pt}<{\centering}|m{13pt}<{\centering}|m{13pt}<{\centering}|m{13pt}<{\centering}|m{13pt}<{\centering}|m{90pt}<{\centering}|}
	\hline
	 \multirow{2}{*}{PDB code} & \multicolumn{10}{c|}{Number of iterations with different $\alpha$} & Energy range
	\\ \cline{2-11}
	 & $0.1$ & $0.2$ & $0.3$ & $0.4$ & $0.5$ & $0.6$ & $0.7$ & $0.8$ & $0.9$ & $1.0$ & (kJ/mol)
	\\ \hline 
	1ay3 (25) & $26$ & $24$ & $22$ & $20$ & $19$ & $18$ & $17$ & $16$ & $15$ & $14$ & $(-271.11,-270.19)$
	\\ \hline
	2olx (65) & $23$ & $19$ & $16$ & $14$ & $12$ & $11$ & $11$ & $10$ & $9$ & $9$ & $(-559.19,-558.09)$
	\\ \hline
	1yjo (121) & $31$ & $24$ & $20$ & $17$ & $15$ & $13$ & $12$ & $11$ & $11$ & $10$ & $(-728.89,-727.31)$
	\\ \hline
	1etn (180) & $33$ & $26$ & $22$ & $19$ & $17$ & $16$ & $14$ & $13$ & $13$ & $12$ & $(-844.68,-842.40)$
	\\ \hline
	ala25 (259) & $30$ & $23$ & $19$ & $16$ & $14$ & $13$ & $12$ & $11$ & $10$ & $9$ & $(-100.53,-100.32)$
	\\ \hline
	1bbl (576) & $29$ & $23$ & $19$ & $17$ & $15$ & $14$ & $13$ & $12$ & $11$ & $10$ & $(-3322.1,-3314.6)$
	\\ \hline

	\multirow{2}{*}{PDB code} & \multicolumn{10}{c|}{Number of iterations with different $\alpha$} & Energy range 
	\\ \cline{2-11}
	 & $1.1$ & $1.2$ & $1.3$ & $1.4$ & $1.5$ & $1.6$ & $1.7$ & $1.8$ & $1.9$ & $2.0$ & (kJ/mol)
	\\ \hline
	1ay3 (25) & $14$ & $13$ & $12$ & $12$ & $11$ & \red $3$ & $5$ & $7$ & $15$ & $28$ & $(-270.88,-270.04)$
	\\ \hline
	2olx (65) & $9$ & $8$ & $8$ & $8$ & $7$ & \red {$5$} & $7$ & $9$ & $11$ & $23$ & $(-558.17,-557.67)$
	\\ \hline
	1yjo (121) & $9$ & $9$ & $8$ & $8$ & $8$ & \red {$3$} & $5$ & $9$ & $11$ & $19$ & $(-728.25,-727.11)$
	\\ \hline
	1etn (180) & $11$ & $11$ & $10$ & $10$ & $9$ & $9$ & $9$ & \red {$5$} & $7$ & $19$ & $(-842.77,-839.17)$
	\\ \hline
	ala25 (259) & $9$ & $8$ & $8$ & $7$ & $7$ & $5$ & \red {$3$} & $9$ & $12$ & $22$ & $(-100.42,-100.29)$
	\\ \hline
	1bbl (576) & $10$ & $9$ & $9$ & $9$ & $8$ & $7$ & \red {$5$} & $7$ & $13$ & $25$ & $(-3315.3,-3313.2)$
	\\ \hline
	\end{tabular}}
\caption{Number of iterations and energy ranges for protein molecules in water ($\varepsilon_1=1,~\varepsilon_2=78.54,~\kappa=0.104$ \AA$^{-1}$). The numbers of atoms are indicated in parentheses next to the PDB codes.}
\label{table:2}
\end{table}
}

\section{Conclusion}
In this paper, we demonstrate the convergence of an interior-exterior nonoverlapping domain decomposition method for the linear Poisson--Boltzmann equation in $\mathbb{R}^3$. Our analysis is nontrivial due to the unboundedness of the exterior subdomain, which distinguishes it from the classical analysis of the Schwarz alternating method with nonoverlapping bounded subdomains.
It is surprisingly good that the convergence condition $0<\alpha<2$ does not depend on $(\varepsilon_1,\varepsilon_2,\kappa,\Omega)$ for the realistic implicit model where $\varepsilon_1\leq \varepsilon_2$.
A sequence of numerical simulations has been conducted to verify our convergence analysis, and to determine the optimal relaxation parameter for the interior-exterior iteration of the generalized ddLPB method.

\section*{Acknowledgement}
 C. Quan is supported by National Natural Science Foundation of China  (Grant No. 12271241), Guangdong Provincial Key Laboratory of Mathematical Foundations for Artificial Intelligence (2023B1212010001), Guangdong Basic and Applied Basic Research Foundation (Grant No. 2023B1515020030), and Shenzhen Science and Technology Innovation Program (Grant No. RCYX20210609104358076). 
C. Quan would like to acknowledge Prof. Martin Gander for encouraging to study this topic during a conference in Suzhou.
C. Quan and X. Liu thank Prof. Huajie Chen for providing financial support for academic communications and discussions. This project has also received funding from the European Research Council (ERC) under the European Union’s Horizon 2020 research and innovation program (grant agreement No. 810367) for Y. Maday.

\appendix
\section{Estimation of \texorpdfstring{$C_{\rm ie}^0$}{} and \texorpdfstring{$C_{\rm ie}^1$}{} for a ball}\label{append}
\setcounter{equation}{0}
\renewcommand{\theequation}{A.\arabic{equation}}
We estimate the newly-defined interior-exterior constants $C_{\rm ie}^0$ and $C_{\rm ie}^1$ for a ball $\Omega = B_{R}(\mathbf 0)$ with radius $R>0$.
Given some $g\neq 0\in H^{\frac 1 2}\big(\partial B_{R}(\mathbf 0)\big) $, we have
	\begin{align}
	u_{\rm r}(r \mathbf s) &= \sum_{\ell=0}^{\infty} \sum_{m=-\ell}^{\ell} c_{\ell,m}\left(
\frac{r}{R}\right)^\ell Y_{\ell}^m(\mathbf s) \quad {\rm for}~ 0\leq r\leq R,~\mathbf s\in
\mathbb{S}^2, 
	\\[1ex]
	u_{\rm c1} (r \mathbf s) & = \sum_{\ell=0}^{\infty} \sum_{m=-\ell}^{\ell} c_{\ell,m}
\frac{k_{\ell}(r)}{k_{\ell}(R)} Y_{\ell}^m(\mathbf s) \quad {\rm for}~ r\geq R,~\mathbf s\in
\mathbb{S}^2, 
	\end{align}
	where $\mathbb{S}^2$ is the boundary of the unit ball, $c_{\ell,m}$ is the
coefficient of spherical harmonics satisfying
	\begin{equation}
	g(R\mathbf s) = \sum_{\ell=0}^{\infty} \sum_{m=-\ell}^{\ell} c_{\ell,m}
Y_{\ell}^{m}(\mathbf s) \quad {\rm for}~ \mathbf s\in \mathbb{S}^2,
	\end{equation}
	and $k_\ell$ is the modified spherical Bessel functions of the second kind \cite[Equation (14.194)]{arfken2012mathematical} with
	\begin{equation}
	k_{\ell}(r) := \sqrt{ \frac{2}{\pi r} } K_{\ell+\frac{1}{2}}(r).
	\end{equation}
	Here, $K_{\alpha}(x)$ with the subscript $\alpha$ is the modified Bessel functions of the second kind. Due to the orthogonality of the spherical harmonics, we calculate
 \begin{equation}
	\begin{aligned}\label{eq:app1}
	&\lVert u_{\rm r} \rVert_{L^2(\Omega)}^2  = \int_0^R  \int_{\mathbb{S}^2} \left( \sum_{\ell=0}^{\infty}
\sum_{m=-\ell}^{\ell} c_{\ell,m}  \left( \frac{r}{R}\right)^{\ell}Y_{\ell}^m(\mathbf s)\right)^2 {\rm d} r \,(r^2{\rm d}\mathbf s)\\
& = \int_0^R r^2 \sum_{\ell=0}^{\infty}
\sum_{m=-\ell}^{\ell} c_{\ell,m} ^2 \left( \frac{r}{R}\right)^{2\ell}  {\rm d} r
= \sum_{\ell=0}^{\infty} \sum_{m=-\ell}^{\ell} c_{\ell,m} ^2
\frac{R^3}{2\ell+3}.
\end{aligned}
\end{equation}
	Multiplying \eqref{eq:uc_1} with $u_{{\rm c1}}$ and integrating over
$\Omega^{\mathsf c}$, we obtain
\begin{equation}
	\begin{aligned}\label{eq:app2}
	&\lVert u_{\rm c1} \rVert_{H^1(\Omega^{\mathsf c})} ^2 = \left< u_{\rm c1},-
\partial_{\mathbf n} u_{\rm c1} \right> 
	\\
	& = \int_{\mathbb{S}^2}  \left( \sum_{\ell=0}^{\infty} \sum_{m=-\ell}^{\ell}
c_{\ell,m} Y_{\ell}^m(\mathbf s)  \right)  \left( - \sum_{\ell'=0}^{\infty}
\sum_{m'=-\ell'}^{\ell'} c_{\ell',m'} \frac{k'_{\ell'}(R)}{k_{\ell'}(R)}
Y_{\ell'}^{m'}(\mathbf s)  \right) (R^2 {\rm d} \mathbf s)
	\\
	& = \sum_{\ell=0}^{\infty} \sum_{m=-\ell}^{\ell} c_{\ell,m}^2   \frac{-
R^2 k'_{\ell}(R)}{k_{\ell}(R)} .
	\end{aligned}
 \end{equation}
The modified spherical Bessel functions have the following properties
\cite[Equation (14.195)]{arfken2012mathematical} :
	\begin{equation}
	k_{\ell-1}(x) - k_{\ell+1}(x) = -\frac{2\ell+1}{x} k_{\ell}(x),\label{eqA8}
	\end{equation}
	\begin{equation}
	\ell k_{\ell-1}(x)+(\ell+1)k_{\ell+1}(x) = -(2\ell+1)k'_{\ell}(x), \quad{\rm
for}~\ell\in\mathbb{N},~x>0.\label{eqA9}
	\end{equation}
    Since $k_0(x) = \frac{e^{-x}}{x},~k_1(x) = e^{-x}\left(\frac{1}{x} + \frac{1}{x^2}\right)$, we know from \eqref{eqA8} and by deduction that $k_{\ell}(x) > 0$. Combining \eqref{eqA8} and \eqref{eqA9}, we then have
    \begin{equation}
	\begin{aligned}
    \frac{- k'_{0}(x)}{k_{0}(x)} & = 1+\frac{1}{x}\geq \frac 1 x, \\
	 \frac{- k'_{\ell}(x)}{k_{\ell}(x)} & = \frac{k_{\ell-1}(x)}{k_{\ell}(x)} + \frac{\ell+1}{x} \geq \frac{\ell+1}{x} \qquad{\rm for}~\ell\in\mathbb{N}_+,~x>0.
	\end{aligned}
     \end{equation}
	Then we can estimate for any $\ell\in\mathbb N$,
	\begin{equation}
 \begin{aligned}
	\frac{{R^3}/{(2\ell+3)}}{{-R^2 k'_{\ell}(R)}/{k_{\ell}(R)} } &\leq \frac{ R^3/ (2\ell+3)}{R (\ell+1)} =  \frac{R^2}{(2\ell+3)(\ell+1)} \leq \frac{R^2}{3}, \\
 \frac{{R^3}/{(2\ell+3)}+R\ell}{{-R^2 k'_{\ell}(R)}/{k_{\ell}(R)} } & \leq \frac{ R^3/ (2\ell+3)+R\ell}{R (\ell+1)} \leq \frac{R^2}{(2\ell+3)(\ell+1)} +1 \leq \frac{R^2}{3}+1.
 \end{aligned}
	\end{equation}
	This together with \eqref{eq:app1} -- \eqref{eq:app2} leads to 
	\begin{equation}
	\lVert u_{\rm r} \rVert_{L^2(\Omega)}^2 \leq \frac{R^2}{3} \lVert u_{\rm c}
\rVert_{H^1(\Omega^{\mathsf c})}^2,\quad \lVert u_{\rm r} \rVert_{H^1(\Omega)}^2 \leq \left(\frac{R^2}{3}+1\right) \lVert u_{\rm c}
\rVert_{H^1(\Omega^{\mathsf c})}^2.
	\end{equation}
implying that
	 \begin{equation}
	 C_{\rm ie}^0 = \sup_{g\neq 0\in H^{\frac 1 2}\big(\partial B_{R}(\mathbf 0)\big)  } \frac{\lVert u_{\rm r} \rVert_{L^2(\Omega)}^2}{ \lVert u_{\rm c}
\rVert_{H^1(\Omega^{\mathsf c})}^2}  \leq \frac{R^2}{3},
	 \end{equation}
  and
	 \begin{equation}
	 C_{\rm ie}^1 = \sup_{g\neq 0\in H^{\frac 1 2}\big(\partial B_{R}(\mathbf 0)\big)  } \frac{\lVert u_{\rm r} \rVert_{H^1(\Omega)}^2}{ \lVert u_{\rm c}
\rVert_{H^1(\Omega^{\mathsf c})}^2}  \leq \frac{R^2}{3}+1.
	 \end{equation}
	 
\section{Spectrum of \texorpdfstring{$(T_{\rm e} + T_{\rm c})^{-1}(T_{\rm r} + T_{\rm c})$}{} for a sphere}\label{append:B}
\setcounter{equation}{0}
\renewcommand{\theequation}{B.\arabic{equation}}
Consider a sphere $\partial B_R(\mathbf 0)$ with radii $R$ and center $\mathbf 0$.
Assume that 
\begin{equation}
g\left(R\mathbf{s}\right)=\sum_{\ell =0}^{\infty} \sum_{m =-\ell }^{\ell }\left[g\right]_{\ell  m }Y_{\ell }^{m }(\mathbf{s}), \quad \mathbf{s} \in \mathbb S^2.
\end{equation}
Then, according to the definitions \eqref{eq:Tr}--\eqref{eq:Te} of DtN operators $T_{\rm r},~ T_{\rm c}$ and $T_{\rm e}$, we have
\begin{equation}
\begin{aligned}
&\partial_{\mathbf{n}} u_{\mathrm{r}}\left(R\mathbf{s}\right)=\sum_{\ell =0}^{\infty} \sum_{m =-\ell }^{\ell }\left[g\right]_{\ell  m }\left(\frac{\ell }{R}\right) Y_{\ell }^{m }(\mathbf{s}), \quad \mathbf{s} \in \mathbb S^2,\\
&\partial_{\mathbf{n}} u_{\mathrm{e}}\left(R\mathbf{s}\right)=\sum_{\ell =0}^{\infty} \sum_{m =-\ell }^{\ell }\left[g\right]_{\ell  m } \frac{i'_{\ell } \left( R\right)}{i_{\ell }\left( R\right)} Y_{\ell }^{m }(\mathbf{s}), \quad \mathbf{s} \in \mathbb S^2,\\
&\partial_{\mathbf{n}} u_{\mathrm{c}}\left(R\mathbf{s}\right)=\sum_{\ell =0}^{\infty} \sum_{m =-\ell }^{\ell }\left[g\right]_{\ell  m } \frac{k'_{\ell } \left( R\right)}{k_{\ell }\left( R\right)} Y_{\ell }^{m }(\mathbf{s}), \quad \mathbf{s} \in \mathbb S^2,
\end{aligned}
\end{equation}
where $i_\ell$ and $k_\ell$ are the spherical Bessel functions of the first and second kinds respectively (see \cite[Equation (14.194)]{arfken2012mathematical}).
Let $\mathcal A = (T_{\text e}+T_{\text c})^{-1}(T_{\text r}+T_{\text c})$.
Since $T_{\rm r} g = \varepsilon_1 \partial_{\mathbf n} u_{\rm r}$, $T_{\rm c} g = -\varepsilon_2 \partial_{\mathbf n} u_{\rm c}$, and $T_{\rm e} g = \varepsilon_2 \partial_{\mathbf n} u_{\rm e}$, we have
\begin{equation}
\mathcal A Y_\ell^m = \left(\frac{i'_{\ell } \left( R\right)}{i_{\ell }\left( R\right)} - \frac{k'_{\ell } \left( R\right)}{k_{\ell }\left( R\right)}\right)^{-1}\left(\frac{\varepsilon_1}{\varepsilon_2}\frac\ell R - \frac{k'_{\ell } \left( R\right)}{k_{\ell }\left( R\right)}\right) Y_\ell^m =: \lambda_{\mathcal A}(\ell) Y_\ell^m.
\end{equation}
As a consequence, the set $\{(\lambda_{\mathcal A}(\ell),Y_\ell^m):\ell\geq 0, -\ell \leq m\leq \ell\}$ collects the eigenpairs of the selfadjoint and positive definite operator $\mathcal A$, and note that $Y_\ell^m$ forms a complete orthogonal basis of $L^2(\mathbb S^2)$.

Consider the case of $\varepsilon_1 = 1,~\varepsilon_2 = 2,~R = 1$. Figure \ref{fig:lambda_A} illustrates the relationship between $\lambda_{\mathcal A}$ and $\ell$, which shows the convergence $\lambda_{\mathcal A}(\ell)\rightarrow \lambda_{\mathcal A}(\infty)$ as $\ell\rightarrow \infty$. Note that $\lambda_{\mathcal A}(\infty)$ is not an eigenvalue of $\mathcal A$ and thus is an element in the continuous spectrum. This indicates that in this simple case, the spectrum of $\mathcal A$ is not only composed of eigenvalues.

\begin{figure}[htb!]
\centering
    \includegraphics[width=0.65\textwidth]{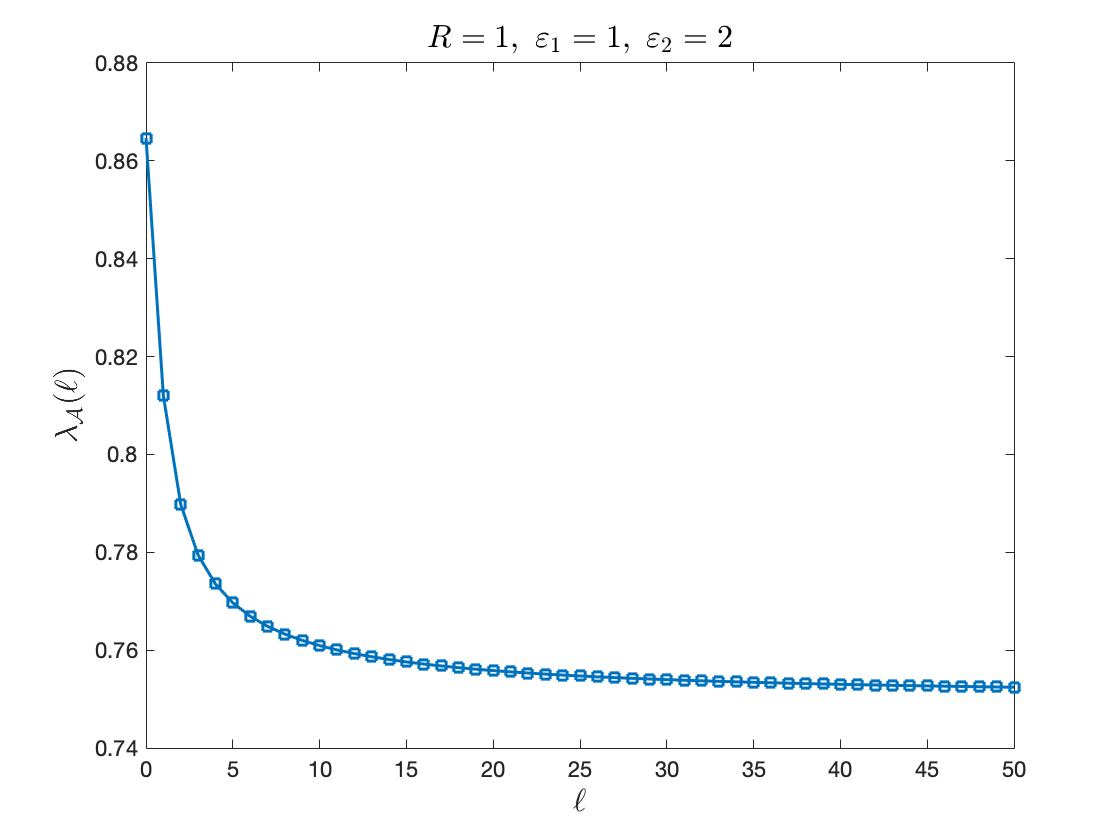}
\caption{$\lambda_{\mathcal A}(\ell)$ v.s. $\ell$ for the unit sphere where $\varepsilon_1=1,~\varepsilon_2=2$.}
\label{fig:lambda_A}
\end{figure}


\bibliography{bibfile}
\bibliographystyle{siam}

\end{document}